\documentclass[a4paper,11pt]{article}
\usepackage[all]{xy}
\title{On o-minimal homotopy groups}

\author{El\'ias Baro \thanks{Partially supported by
GEOR MTM2005-02568}\\
Departamento de Matem\'aticas\\ Universidad Aut\'onoma de Madrid\\ 28049 Madrid, Spain \and Margarita Otero\thanks{Partially supported by
GEOR MTM2005-02568 and Grupos UCM 910444}\\ Departamento de Matem\'aticas\\ Universidad Aut\'onoma de Madrid\\28049 Madrid, Spain}

\usepackage{amsmath}
\usepackage{amscd}
\usepackage{amsfonts}
\usepackage{amssymb}
\usepackage{amsthm}

\newtheorem{deff}{Definition}[section]
\newtheorem{teo}[deff]{Theorem}
\newtheorem{lema}[deff]{Lemma}
\newtheorem{cor}[deff]{Corollary}
\newtheorem{prop}[deff]{Proposition}
\newtheorem{fact}[deff]{Fact}
\newtheorem{ejemcur}[deff]{Example}
\theoremstyle{definition}

\newtheorem{nota}[deff]{Notation}

\newtheorem{obs}[deff]{Remark}

\newcommand{\LD}{ld-space}
\newcommand{\ld}{LD-space}

\newcommand{\bt}{\begin{theorem}}
\newcommand{\et}{\end{theorem}}
\newcommand{\bl}{\begin{lemma}}
\newcommand{\el}{\end{lemma}}
\newcommand{\bexa}{\begin{example}}
\newcommand{\eexa}{\end{example}}
\newcommand{\bexe}{\begin{exercise}}
\newcommand{\eexe}{\end{exercise}}
\newcommand{\bprop}{\begin{proposition}}
\newcommand{\eprop}{\end{proposition}}
\newcommand{\bp}{\begin{proof}}
\newcommand{\ep}{\end{proof}}
\newcommand{\bc}{\begin{corollary}}
\newcommand{\ec}{\end{corollary}}
\newcommand{\bd}{\begin{definition}}
\newcommand{\ed}{\end{definition}}
\newcommand{\br}{\begin{remark}}
\newcommand{\er}{\end{remark}}
\begin{document}
\textbf{\begin{center}LOCALLY DEFINABLE HOMOTOPY\end{center}}
\begin{center}\begin{footnotesize}EL\'IAS BARO \footnote{Partially supported by
GEOR MTM2005-02568.}\begin{scriptsize} AND  \end{scriptsize}MARGARITA OTERO\footnote{Partially supported by
GEOR MTM2005-02568 and Grupos UCM 910444.

\textit{Date}: November 16, 2008

\textit{Mathematics Subject Classification 2000}: 03C64, 14P10, 55Q99.}
\end{footnotesize}\end{center}
\begin{quote}
\begin{footnotesize}\begin{scriptsize}ABSTRACT\end{scriptsize}. In \cite{07pBO} o-minimal homotopy was developed  for the definable category, proving o-minimal versions of the Hurewicz theorems and the Whitehead theorem. Here, we extend these results to the category of locally definable spaces, for which we introduce homology and homotopy functors. 
We also study the concept of connectedness in $\bigvee$-definable groups -- which are examples of locally definable spaces. We show that the various concepts of connectedness associated to these groups, which have  appeared in the literature, are non-equivalent.
\end{footnotesize} 
\end{quote}

\section{Introduction}\label{sintro}
According to H.Delfs and M.Knebusch, the reference \cite{85DK} is the first part of what ``it is designed as a \textit{topologie g\'{e}n\'{e}rale} for semialgebraic geometry''. The main purpose of the book is to introduce a new category extending the semialgebraic one and large enough to be able to deal with objects such as covering maps of ``infinite degree''. Specifically, the authors define locally semialgebraic spaces, roughly, as those obtained by glueing infinitely many affine semialgebraic sets. 

In the o-minimal setting we have the corresponding  situation, the definable category is not large enough to deal with certain natural objects. Even though the theory of locally semialgebraic spaces had not been formally extended to the o-minimal framework, some related notions have already appeared -- always carrying a group structure. This is the case of $\bigvee$-definable groups which were used by Y. Peterzil and S. Starchenko in \cite{00PS} as a tool for the study of interpretability problems. Later, M. Edmundo introduces a restricted notion of $\bigvee$-definable groups in \cite{06E} and he develops a whole theory around them. However, the latter two categories are not so flexible and general as the locally definable category. For instance, in the locally definable category there is a natural adaptation of the classical construction of universal coverings which generalize the corresponding result for restricted $\bigvee$-definable groups in \cite{06EEl}. Another example of the rigidity of the $\bigvee$-definable groups and their restricted analogues are the \textit{non-equivalently} notions of connectedness introduced in \cite {06E},\cite{06EEl},\cite{08OP} and \cite{00PS} which we can now clarify  by considering the locally definable category.

On the other hand, in \cite{85DK}, after introducing the locally semialgebric category, locally semialgebraic homotopy theory is developed. Delfs and Knebusch first prove -- using the Tarski-Seidenberg principle -- some beautiful results relating both the semialgebraic and the classical homotopy of semialgebraic sets defined without parameters -- and hence realizable over the reals. Then, they generalize these results to regular paracompact locally semialgebraic spaces -- the nice ones. Because of the lack of the Tarski-Seidenberg principle in o-minimal structures, only  the o-minimal fundamental group was considered (see \cite{02BeO}) with strong consequences in the study of definable groups in \cite{04EO}. In \cite{07pBO}, the authors fill this gap -- in the study of definable homotopy -- by relating both the o-minimal and the semialgebraic (higher) homotopy groups. The core of the latter work is the adaptation to the o-minimal setting of some techniques used in  \cite{85DK} via a refinement of the Triangulation theorem (see the Normal Triangulation Theorem in \cite{07pB}).

Having at hand these recent results for the o-minimal homotopy theory, it seems to us natural to extend them to the locally definable category. Therefore, we have taken this opportunity to develop the locally definable category in o-minimal structures expanding a real closed field. Furthermore, we have tried to unify the related notions of $\bigvee$-groups and their restricted version via the theory of locally definable spaces. We also point out that we have avoided the presentation style of Delfs and Knebusch in \cite{85DK} with ``sheaf'' flavour, using instead the natural generalization of definable spaces of L.van den Dries in \cite{98Dr}.

The results of this paper have already been applied to prove the contractibility of the universal covering group of a definably compact abelian group (see \cite{08pBMO}).

In section \ref{sec:definition} we first introduce the category of locally definable spaces (in short ld-spaces). Locally definable spaces of special interest are the regular paracompact ones (in short LD-spaces). We collect the relevant facts from \cite{85DK} which can be directly adapted to our context, most notably the Triangulation Theorem for LD-spaces (for completeness we have included a proof of this last result in an appendix). A homology theory for LD-spaces is developed in Section \ref{shomol} via an alternative approach to that of \cite{85DK}  for locally semialgebraic spaces (which goes through sheaf cohomology). In \cite{00PS} it is implicitly proved that the $\bigvee$-definable groups are examples of ld-spaces, in Section \ref{sexamples} we prove that the restricted ones are moreover paracompact --and hence LD-spaces-- and we also discuss other examples of ld-spaces. In Section \ref{sec:connectedness} we deal with connectedness for ld-spaces and  we will clarify the relation among the different notions of connectedness used for $\bigvee$-definable groups which appear in the literature, pointing out the inadequacy of some of them. Finally, with all these tools at hand,  we prove in Section \ref{shomot} the generalizations to LD-spaces of the homotopy results in \cite{07pBO}, in particular the Hurewicz theorems and the Whitehead theorem.

The results of this paper are part of the first author's Ph.D. dissertation.
\section{Preliminaries on locally definable spaces}\label{sec:definition}
We fix an o-minimal expansion $\mathcal{R}$ of a real closed field $R$. We take the order topology on $R$ and the product topology on $R^n$ for $n>1$. For the rest of this paper, ``definable'' means ``definable with parameters'' and ``definable map'' means ``continuous definable map'', unless otherwise specified.

We shall briefly discuss the category of locally definable spaces. All the results we list in this Section are analogous to those of locally semialgebraic spaces in \cite{85DK}. The proofs of these results in \cite{85DK} are based on properties of semialgebraic sets which are shared by definable sets. Hence we have not included their proofs here.
\begin{deff}Let $M$ be a set. An \textbf{atlas} on $M$ is a family of \textbf{charts} $\{(M_i,\phi_i)\}_{i\in I}$, where $M_i$ is a subset of $M$ and  $\phi_i:M_i\rightarrow Z_i$ is a bijection between $M_i$ and a definable set $Z_i$ of $R^{n(i)}$ for all $i\in I$, such that $M=\bigcup_{i\in I}M_i$ and for each pair $i,j\in I$ the set $\phi_i(M_i\cap M_j)$ is a relative open definable subset of $Z_i$ and the map
$$\phi_{ij}:=\phi_j\circ \phi^{-1}_i:\phi_i(M_i\cap M_j)\rightarrow M_i\cap M_j \rightarrow \phi_j(M_i\cap M_j)$$
is definable. We say that $(M,M_i,\phi_i)_{i\in I}$ is a \textbf{locally definable space}. The \textbf{dimension} of $M$ is $dim(M):=sup\{dim(Z_i):i\in I\}$. If $Z_i$ and $\phi_{ij}$ is defined over $A$ for all $i,j\in I$, $A\subset R$, we say that $M$ is a locally definable space over $A$.

We say that two atlases $(M,M_i,\phi_i)_{i\in I}$ and $(M,M'_j,\psi_j)_{j\in J}$ on a set $M$ are \textbf{equivalent} if and only if for all $i\in I$ and $j\in J$ we have that \emph{(i)} $\phi_{i}(M_i\cap M'_j)$ and $\psi_{j}(M_i\cap M'_j)$ are relative open definable subsets of $\phi_i(M_i)$ and $\psi_j(M'_j)$ respectively, \emph{(ii)} the map $\psi_j\circ \phi_i^{-1}|_{\phi_{i}(M_i\cap M'_j)}:\phi_{i}(M_i\cap M'_j)\rightarrow M_i\cap M'_j \rightarrow \psi_{j}(M_i\cap M'_j)$ and its inverse are definable and \emph{(iii)} $M_i\subset \bigcup_{k\in J_0}M'_k$ and $M'_j\subset \bigcup_{s\in I_0}M_s$ for some finite subsets $J_0$ and $I_0$ of $J$ and $I$ respectively.
\end{deff}
Note that in the above definition if we take $I$ to be finite then $M$ is just a definable space in the sense of \cite{98Dr}. In fact, some of the notions that we are going to introduce in this section are generalizations of the corresponding  ones in the category of definable spaces.

Even though the above definition seems different from its semialgebraic analogue (see \cite[Def.I.3]{85DK}), they are actually equivalent. In \cite{85DK} it is (implicitly) proved that Definition I.3 is equivalent to the semialgebraic analogue of our definition here (see \cite[Lem.I.2.2]{85DK} and the remark after \cite[Lem.I.2.1]{85DK}). The same proofs can be adapted to the o-minimal setting.

Given a locally definable space $(M,M_i,\phi_i)$, there is a unique topology in $M$ for which $M_i$ is open and $\phi_i$ is a homeomorphism for all $i\in I$. For the rest of the paper any topological property of locally definable spaces refers to this topology. We are mainly interested in Hausdorff topologies. \textit{Henceforth, an \textbf{\LD} means a Hausdorff locally definable space.}
\vspace{0.4cm}

We now introduce the subsets of interest in the category of \LD s.
\begin{deff}\label{deff:definable}Let $(M,M_i,\phi_i)_{i\in I}$ be an \LD. We say that a subset $X$ of $M$ is a \textbf{definable subspace} of $M$ (over $A$) if there is a finite $J\subset I$ such that $X\subset \bigcup_{j\in J}M_j$ and $\phi_{j}(M_{j}\cap X)$ is definable (resp. over $A$) for all $j\in J$. A subset $Y\subset M$ is an \textbf{admissible subspace} of $M$ (over $A$) if $\phi_i(Y\cap M_i)$ is definable (resp. over $A$) for all $i\in I$, or equivalently, $Y\cap X$ is a definable subspace of $M$ (resp. over $A$) for every definable subspace $X$ of $M$ (resp. over $A$).
\end{deff}
The admissible subspaces of an \LD \ are closed under complements, finite unions and finite intersections. Moreover, the interior and the closure of an admissible subspace is an admissible subspace.

Every definable subspace of an \LD \ is admissible. The definable subspaces of an \LD \ are closed under finite unions and finite intersections, but not under complements. The interior of a definable subspace is a definable subspace. However, the closure of a definable subspace might not be  a definable subspace (see Example \ref{expara}).
\begin{obs}\label{obs:admssiblesets}Given an \LD \ $(M,M_i,\phi_i)_{i\in I}$ we have that every admissible subspace $Y$ of $M$ inherits in a natural way a structure of an \LD, whose atlas is $(Y,Y_i,\psi_i)_{i\in I}$, where $Y_i:=M_i\cap Y$ and $\psi_i:=\phi_i|_{Y_i}$. In particular, if $Y$ is a definable subspace then it inherits the structure of a definable space.
\end{obs}
Now, we introduce the maps that we will use in the locally definable category. First, note that given two \LD s $M$ and $N$, with their atlas $(M_i, \phi_i)_{i\in I}$ and
$(N_j, \psi_j)_{j\in J}$, respectively, the atlas $(M_i\times N_j, (\phi_i, \psi_j))_{i\in I,j\in J}$ makes $M\times N$ into an \LD. In particular, if $M$ and $N$ are definable spaces, then $M\times N$ is a definable space. Recall that a map $f$ from a definable space $M$ into a definable space $N$ is a definable map over $A$, $A\subset R$, if its graph is a definable subset of $M\times N$ over $A$.
\begin{deff}\label{def:locmap}Let $(M,M_i,\phi_i)_{i\in I}$ and $(N,N_j,\phi_j)_{j\in J}$ be \LD s. A map $f:M  \rightarrow N$ is an \textbf{ld-map} over $A$, $A\subset R$, if $f(M_i)$ is a definable subspace of $N$ and the map $f|_{M_i}:M_i\rightarrow f(M_i)$ is definable over $A$ for all $i\in I$.
\end{deff}
The behavior of admissible subspaces and ld-maps in the locally definable category is different from that of definable subsets and definable maps in the definable category. For, even though the preimage of an admissible subspace by an ld-map is an admissible subspace, the image of an admissible subspace by an ld-map might not be an admissible subspace (see comments after Example \ref{ex}). Nevertheless, the image of a definable subspace by an ld-map is a definable subspace. In particular, let us note that every ld-map between definable spaces is a definable map and therefore the category of definable spaces is a full subcategory of the category of \LD s. On the other hand, given two \LD s $M$ and $N$, the graph of an ld-map $f:M\rightarrow N$ is an admissible subspace of $M\times N$. However, not every continuous map $f:M\rightarrow N$ whose graph is admissible in $M\times N$ is an ld-map.

\vspace{0.4cm}
The notion of connectedness in the locally definable category  which we now introduce  is a subtle issue. It extends the natural concept of ``definably connected'' for definable spaces. In Section \ref{sec:connectedness} below we will analyze this concept and we will compare it with other definitions introduced by different authors in the study of $\bigvee$-groups.
\begin{deff}Let $M$ be an \LD \ and $X$ an admissible subspace of $M$. We say that $X$ is \textbf{connected} if there is no admissible subspace $U$ of $M$ such that $X\cap U$ is both open and closed in $X$.
\end{deff}

We now introduce \LD s with some special properties. As we will see below, in the \LD s with these properties there is a good relation between both the topological and the definable setting. Moreover, they form an adequate framework to develop a homotopy theory.

We say that an \LD \ $(M,M_i,\phi_i)$ is \textit{\textbf{regular}} if every $x\in M$ has a fundamental system of closed (definable) neighbourhoods, i.e., for every open $U$ of $M$ with $x\in U$ there is a closed (definable) subspace $C$ of $M$ such that $C\subset U$ and $x\in int(C)$. Equivalently, an \LD \ $M$ is regular if for every closed subset $C$ of $M$ and every point $x\in M\setminus C$ there are open (admissible) disjoint subsets $U_1$ and $U_2$ with $C\subset U_1$ and $x\in U_2$.
\begin{obs}\label{obs:regularity}If $M$ is a regular \LD \ then every definable subspace of $M$ can be interpreted as an affine set, i.e, as a definable set of $R^n$ for some $n\in \mathbb{N}$. For, suppose that $X$ is a definable subspace of $M$. Then, $X$ inherits a structure of definable space from $M$ (see Remark \ref{obs:admssiblesets}). Since $M$ is regular then $X$ is also regular. Finally, by the o-minimal version of Robson's embedding theorem, $X$ is affine (see \cite[Ch.8,Thm. 1.8]{98Dr}).
\end{obs}
Let $(M,M_i,\phi_i)_{i\in I}$ be an \LD. A family $\{X_j\}_{j\in J}$ of admissible subspaces of $M$ is an \textit{\textbf{admissible covering}} of $X:=\bigcup_{j\in J}X_j$ if for all $i\in I$, $M_i\cap X=M_i\cap (X_{j_1}\cup \cdots \cup X_{j_l})$ for some $j_1,\ldots,j_l\in J$ (note that in particular $X$ is an admissible subspace). A family $\{Y_j\}_{j \in J}$ of admissible subspaces of $M$ is \textit{\textbf{locally finite}} if for all $i\in I$ we have that $M_i\cap Y_j\neq \emptyset$ for only a finite number of $j\in J$ (note that in particular it is an admissible covering of their union). In general, not every admissible covering is locally finite (see Example \ref{expara}). We say that an \LD \ $M$ is \textit{\textbf{paracompact}} if there exists a locally finite covering of $M$ by open definable subspaces. Note that this notion is ``weaker'' than the classical one. It is easy to prove that if $M$ is paracompact then every admissible covering of $M$ has a locally finite refinement (see \cite[Prop. I.4.5]{85DK}). We say that an \LD \ $M$ is \textit{\textbf{Lindel\"of}} if there exist an admissible covering of $M$ by countably many open definable subspaces. Paracompactness provide us a good relation between both the topological and definable setting.
\begin{fact}\label{fact:claudef}Let $M$ be an \LD.\\
\emph{(1)} \emph{\cite[Prop. I.4.6]{85DK}} If $M$ is paracompact then for every definable subspace $X$, the closure $\overline{X}$ is also a definable subspace of $M$.\\
\emph{(2)} \emph{\cite[Thm. I.4.17]{85DK}} If $M$ is connected and paracompact then $M$ is Lindel\"of.\\
\emph{(3)} \emph{\cite[Prop. I.4.18]{85DK}} If $M$ is Lindel\"of and for every definable subspace $X$ its closure $\overline{X}$ is also a definable subspace, then $M$ is paracompact.
\end{fact}
\begin{proof}The locally definable versions of the proofs of the above facts are just adaptations of the semialgebraic ones. Nevertheless, we prove here (3) to give an idea of the flavour of these proofs. Let $\{M_n:n\in \mathbb{N}\}$ be an admissible covering of $M$ by countably many open definable subspaces. We can assume that $M_n\subset M_{n+1}$ for every $n\in \mathbb{N}$. Moreover, since each $\overline{M_n}$ is also a definable subspace, we can assume that $M_n\subset \overline{M_n}\subset M_{n+1}$ for every $n\in \mathbb{N}$. Consider the family $U_0=M_0$, $U_1=M_1$, $U_n=M_n\setminus \overline{M_{n-2}}$ for every $n\geq 2$. Then, $\{U_n:n\in \mathbb{N}\}$ is a locally finite covering of $M$ by open definable subspaces. 
\end{proof}

The fact that definable subspaces are affine together with paracompactness permits to establish a Triangulation Theorem for regular and paracompact \LD s (which will be essential for the proof of the Hurewicz and Whitehead theorems below). Fix a cardinal $\kappa$. We denote by $R^{\kappa}$ the $R$-vector space generated by a fixed basis of cardinality $\kappa$. A \textit{\textbf{generalized simplicial complex}} $K$ in $R^{\kappa}$ is a usual simplicial complex except that we may have infinitely many (open) simplices. The \textit{\textbf{locally finite}} generalized simplicial complexes are those ones for which the star of each simplex is a finite subcomplex. On them we can define in an obvious way an \LD \ structure. Indeed, given a locally finite generalized simplicial complex $K$, for each $\sigma\in K$ we have that $St_{_K}(\sigma)$ is a finite subcomplex and therefore $St_{_K}(\sigma)\subset R^{n_{\sigma}}\subset R^{\kappa}$ for some $n_{\sigma}\in \mathbb{N}$. Now, giving each $St_{_K}(\sigma)$ the topology it inherits from $R^{n_{\sigma}}$, it suffices to consider the atlas $\{(St_{_K}(\sigma),id|_{St_{_K}(\sigma)} \}_{\sigma\in K}$. With this \LD \ structure, a locally finite generalized simplicial complex is regular and paracompact. The next fact is a sort of converse of the last statement. 
\begin{fact}[Triangulation Theorem]\emph{\cite[Thm. II.4.4]{85DK}}\label{fact:triangulation} Let $M$ be a regular and paracompact \LD \ and let $\{A_j:j\in J\}$ be a locally finite family of admissible subspaces of $M$. Then, there exists an ld-triangulation of $M$ partitioning $\{ A_j:j\in J\}$, i.e., there is a locally finite generalized simplicial complex $K$ and a ld-homeomorphism $\psi:|K|\rightarrow M$, where $|K|$ is the realization of $K$, such that $\psi^{-1}(A_j)$ is the realization of a subcomplex of $K$ for every $j\in J$.
\end{fact}
\begin{obs}\label{rmk:paratriang} In the Triangulation theorem above, and as in the definable case, we can find the generalized simplicial complex $K$ with its vertices tuples of real algebraic numbers. For, as in the classical theory, if we consider $K$ as an abstract complex and we denote by $\kappa$ the cardinal of the set of vertices of $K$, then we obtain a ``canonical realization'' of $K$ in $R^{\kappa}$ whose vertices are the standard basis of $R^{\kappa}$.  Moreover, if the \LD \ $M$ is defined over $A$, $A\subset R$, then we can find the locally definable homeomorphism $\psi$ defined over $A$.
\end{obs}
As before, the proof of the above fact is just an adaptation of the semialgebraic one. However, because of the relevance of this result, we have included here a sketch of the proof for completeness (see Appendix \ref{sappepr}). Let us note that the hardest part of this proof, which may be of interest by itself, is to show that we can embed an \ld \ in another one with good properties. We say that an \LD \ $M$ is \textit{\textbf{partially complete}} if every closed definable subspace $X$ of $M$ is definably compact, i.e., every definable curve in $X$ is completable in $X$. 
\begin{fact}\emph{\cite[Thm. II.2.1]{85DK}}\label{fembed} Let $M$ be an \ld. Then, there exist an embedding of $M$ into a partially complete \ld, i.e, there is partially complete \LD \ $N$ and a ld-map $i:M\rightarrow N$ such that $i(M)$ is an admissible subspace of $N$ and $i:M\rightarrow i(M)$ is an ld-homeomorphism (where $i(M)$ has the \ld \ structure inherited from $M$).
\end{fact}

\textit{Henceforth, we denote a regular and paracompact \LD \ by} \textit{\textbf{\ld}}. Note that by Fact \ref{fact:claudef}.(2) a connected \ld \ is Lindel\"of.\\

We finish this section studying the behavior of \LD s with respect to model theoretic operators. Firstly, let us show that given an elementary extension $\mathcal{R}_1$ of an o-minimal structure $\mathcal{R}$ and given an \LD \ $M$ in $\mathcal{R}$, there is a natural \textit{\textbf{realization}} $M(\mathcal{R}_1)$ of $M$ over $\mathcal{R}_1$. For, denote by $\{\phi_i:M_i \rightarrow Z_i\}_{i\in I}$ the definable atlas of $M$ and consider the set $Z=\bigcup_{i\in I}Z_i/\sim$, where $x\sim y$ for $x\in Z_i$ and $y\in Z_j$ if and only if $\phi_{ij}(x)=y$. Note that we can define an \LD \ structure on $Z$ in a natural way and that $Z$ with this \LD \ structure is isomorphic to $M$ (see Definition \ref{def:locmap}). Now, the realization $Z(R_1)$ is just $\bigcup_{i\in I}Z_i(R_1)$ modulo the relation $\sim_{\mathcal{R}_1}$, where $x\sim_{\mathcal{R}_1} y$ for $x\in Z_i(R_1)$ and $y\in Z_j(R_1)$ if and only if $\phi_{ij}(\mathcal{R}_1)(x)=y$. On the other hand, note that given an o-minimal expansion $\mathcal{R}'$ of $\mathcal{R}$ and an \LD \ $M$ in $\mathcal{R}$, we can consider $M$ as an \LD \ in $\mathcal{R}'$.
\begin{prop}\label{prop:epanext}Let $\mathcal{R}'$ be an o-minimal expansion of $\mathcal{R}$ and let $\mathcal{R}_1$ be an elementary extension of $\mathcal{R}$. Let $M$ be an \LD \ in $\mathcal{R}$. Then,\\
\emph{(i)} $M$ is a connected \ld \ if and only if $M(\mathcal{R}_1)$ is a connected \ld ,\\
\emph{(ii)}  $M$ is regular in $\mathcal{R}$ if and only if it is regular in $\mathcal{R}'$,\\
\emph{(ii)}  $M$ is connected in $\mathcal{R}$ if and only if it is connected in $\mathcal{R}'$,\\
\emph{(iii)} $M$ is Lindel\"of in $\mathcal{R}$ if and only if it is Lindel\"of in $\mathcal{R}'$,\\
\emph{(iv)} $M$ is paracompact in $\mathcal{R}$ if and only if it is paracompact in $\mathcal{R}'$.
\end{prop}
\begin{proof}(i) follows from the Triangulation Theorem (see Fact \ref{fact:triangulation}). (ii) is trivial and (iii) can be easily deduced from Fact \ref{fact:conexoporarcos}. Let us show that if $M$ is Lindel\"of in $\mathcal{R}'$ then $M$ is Lindel\"of in $\mathcal{R}$ (the converse is trivial). Indeed, let $(M_i,\phi_i)_{i\in I}$ be an atlas of $M$ in $\mathcal{R}$ and $\{U_n:n\in \mathbb{N}\}$ be a countable admissible covering  of $M$ by open definable subspaces in $\mathcal{R}'$ of $M$. Since each $U_n$ is a definable subspace, it is contained in a finite union of charts $M_i$. Therefore, there exists a countable subcovering of $\{M_i:i\in I\}$ which already covers $M$ and hence $M$ is Lindel\"of in $\mathcal{R}$. Now, we show that if $M$ is paracompact in $\mathcal{R}'$ then $M$ is paracompact in $\mathcal{R}$ (the converse is trivial). Suppose $M$ is paracompact in $\mathcal{R}'$. Without loss of generality we can assume that $M$ is connected. Therefore, by the above equivalences and Fact \ref{fact:claudef}.(2), $M$ is Lindel\"of in $\mathcal{R}$. Then, by Fact \ref{fact:claudef}.(3), it suffices to prove that for every definable subspace $X$ of $M$ in $\mathcal{R}$, its closure $\overline{X}$ is also a definable subspace of $M$ in $\mathcal{R}$. Since $M$ is paracompact in $\mathcal{R}'$, the latter is clear by Fact \ref{fact:claudef}.(1). 
\end{proof}
\section{Homology of locally definable spaces}\label{shomol}
We fix for the rest of this section an \ld \ $M$. We consider the abelian group $S_k(M)^{\mathcal{R}}$ freely generated by the \textit{singular locally definable simplices} $\sigma:\Delta_k\rightarrow M$, where $\Delta_k$ is the standard $k$-dimensional simplex in $R$. Note that since $\sigma$ is locally definable and $\Delta_k$ is definable, the image $\sigma(\Delta_k)$ is a definable subspace of $M$. As we will see, this fact allows us to use the o-minimal homology developed by A. Woerheide in \cite{96W} (see also \cite{07pB} for an alternative development of simplicial o-minimal homology). The boundary operator $\delta:S_{k+1}(M)^{\mathcal{R}}\rightarrow S_{k}(M)^{\mathcal{R}}$ is defined as in the classical case, making $S_*(M)^{\mathcal{R}}=\bigoplus_k S_k(M)^{\mathcal{R}}$ into a chain complex. We similarly define the chain complex of a pair of locally definable spaces. The graded group $H_*(M)^{\mathcal{R}}=\bigoplus_k H_k(M)^{\mathcal{R}}$ is defined as the homology of the complex $S_*(M)^{\mathcal{R}}$. Locally definable maps induce in a natural way homomorphisms in homology. Similarly for relative homology. Note that if $M$ is just a definable set then we obtain the usual o-minimal homology groups (see e.g. \cite{04EO}).

It remains to check that the functor we have just defined satisfies the locally definable version of the Eilenberg-Steenrod axioms. We shall check them making use of the corresponding axioms for definable sets through an adaptation of a classical result in homology that (roughly) states that the homology commutes with direct limits. Note that each definable subspace $Y\subset M$ is a definable regular space and hence affine (see Remark \ref{obs:regularity}). Therefore, the o-minimal homology groups of $Y$ as definable set are the ones we have just defined as (locally) definable space. Denote by $\mathcal{D}_M$ the set $$\{Y\subset M:Y \textrm{ definable subspace}\}.$$ Note that $M$ can be written as the directed union $M=\bigcup_{Y\in \mathcal{D}_M}Y$. Now, consider the direct limit $$\underrightarrow{lim}_{_{Y\in \mathcal{D}_M}} H_n(Y)^{\mathcal{R}}= \bigcup\hspace{-3.3mm}\cdot\hspace{2mm}_{_{Y\in \mathcal{D}_M}} H_{n}(Y)^{\mathcal{R}}/\sim,$$ where $c_1\sim c_2$ for $c_1\in H_n(Y_1)^{\mathcal{R}}$ and $c_2\in H_n(Y_2)^{\mathcal{R}}$, $Y_1,Y_2\in \mathcal{D}_M$, if and only if there is $Y_3\in \mathcal{D}_M$ with $Y_1,Y_2\subset Y_3$ such that $(i_1)_*(c_1)=(i_2)_*(c_2)$ for $(i_1)_*:H_n(Y_1)^{\mathcal{R}}\rightarrow H_n(Y_3)^{\mathcal{R}}$ and $(i_2)_*:H_n(Y_2)^{\mathcal{R}}\rightarrow H_n(Y_3)^{\mathcal{R}}$ are the homomorphisms in homology induced by the the inclusions. Then, we have a well-defined homomorphism $(i_{_Y})_*:H_n(Y)^{\mathcal{R}}\rightarrow H_n(M)^{\mathcal{R}}$ for each $Y\in \mathcal{D}_M$, where $i_{_Y}:Y \rightarrow M$ is the inclusion. Hence, there exists a well-defined homomorphism $$\psi:\underrightarrow{lim}_{_{Y\in \mathcal{D}_M}} H_n(Y)^{\mathcal{R}}\rightarrow H_n(M)^{\mathcal{R}},$$ where $\psi(\overline{c})=(i_{_Y})_*(c)$ for $c\in H_n(Y)^{\mathcal{R}}$.  In a similar way, given an admissible subspace $A$ of $M$, we have a well-defined homomorphism  $$\widetilde{\psi}:\underrightarrow{lim}_{_{Y\in \mathcal{D}_M}} H_n(Y,A\cap Y)^{\mathcal{R}}\rightarrow H_n(M,A)^{\mathcal{R}},$$ where $\widetilde{\psi}(\overline{c})=i_*(c)$ for $c\in H_n(Y,A\cap Y)^{\mathcal{R}}$ and $i:(Y,Y\cap A)\rightarrow (M,A)$ the inclusion map.
\begin{teo}\label{teo:limdirhom}\emph{(i)} $\psi:\underrightarrow{lim}_{_{Y\in \mathcal{D}_M}}H_n(Y)^{\mathcal{R}}\rightarrow H_n(M)^{\mathcal{R}}$ is an isomorphism.\\
\emph{(ii)} Let $A$ be an admissible subspace of $M$. Then $\widetilde{\psi}:\underrightarrow{lim}_{_{Y\in \mathcal{D}_M}} H_n(Y,A\cap Y)^{\mathcal{R}}\rightarrow H_n(M,A)^{\mathcal{R}}$ is an isomorphism.
\end{teo}
\begin{proof}(i) Firstly, we show that $\psi$ is surjective. Let $c\in H_n(M)^{\mathcal{R}}$ and $\alpha$ be a finite sum of singular ld-simplices of $M$ which represents $c$. Consider the definable subspace $X$ of $M$ which is the union of the images of the singular ld-simplices in $\alpha$. Hence $[\alpha]\in H_n(X)^{\mathcal{R}}$ and therefore it suffices to consider $\overline{[\alpha]}\in \underrightarrow{lim}_{_{Y\in \mathcal{D}_M}}H_n(Y)^{\mathcal{R}}$. Now, let us show that $\psi$ is injective. Let $\overline{c}\in \underrightarrow{lim}_{_{Y\in \mathcal{D}_M}}H_n(Y)^{\mathcal{R}}$, $c\in H_n(X)^{\mathcal{R}}$, $X\in \mathcal{D}_M$, such that $\psi(\overline{c})=0$. Since $\psi(\overline{c})=0$, there is a finite sum $\beta$ of singular ld-simplices of $M$ such that $\delta\beta=\alpha$. Consider the definable subspace $Z$ of $M$ which is the union of $X$ and the images of the singular ld-simplices in $\beta$. Then we have that $[\alpha]=0$ in $H_n(Z)^{\mathcal{R}}$ and therefore $\overline{c}=0$ in $\underrightarrow{lim}_{_{Y\in \mathcal{D}_M}}H_n(Y)^{\mathcal{R}}$. The proof of (ii) is similar.
\end{proof}
\begin{obs}\label{obs:directlimit}Let $M$ be an \ld \ and $\mathcal{D}$ a collection of definable subspaces of $M$ such that for every $Y\in \mathcal{D}_M$ there is $X\in \mathcal{D}$ with $Y\subset X$. Then Theorem \ref{teo:limdirhom} remains true if we replace $\mathcal{D}_M$ by $\mathcal{D}$.
\end{obs}
Now, with the above result, we verify the Eilenberg-Steenrod axioms.
\begin{prop}[Homotopy axiom] Let $M$ and $N$ be \ld s and let $A$ and $B$ be admissible subspaces of $M$ and $N$ respectively. If $f:(M,A)\rightarrow (N,B)$ and $g:(M,A)\rightarrow (N,B)$ are ld-homotopic ld-maps then $f_*=g_*$.
\end{prop}
\begin{proof}Let $[\alpha]\in H_n(M,A)^{\mathcal{R}}$. Consider the definable subspace $X$ of $M$ which is the union of the images of the singular ld-simplices in $\alpha$. By Theorem \ref{teo:limdirhom} and the homotopy axiom for definable sets, it is enough to prove that there is a definable subspace $Z$ of $N$ such that $f(X),g(X)\subset Z$ and that the definable maps $f|_X:(X,A\cap X)\rightarrow (Z,B\cap Z)$ and $g|_X:(X,A\cap X)\rightarrow (Z,B\cap Z)$ are definably homotopic. Let $F:(M\times I,A\times I) \rightarrow (N,B)$ be a ld-homotopy from $f$ to $g$. Then, it suffices to take $Z$ as the definable subspace $F(X\times I)$ of $N$ and the definable homotopy $F|_{X\times I}:(X\times I,A\cap X \times I) \rightarrow (Z,B\cap Z)$ from $f|_X$ to $g|_X$.\end{proof}
\begin{prop}[Exactness axiom]Let $A$ be an admissible subspace of $M$ and let $i:(A,\emptyset)\rightarrow (M,\emptyset)$ and $j:(M,\emptyset)\rightarrow (M,A)$ be the inclusions. Then the following sequence is exact
$$\cdots \rightarrow H_n(A)^{\mathcal{R}}\stackrel{i_*}{\rightarrow} H_n(M)^{\mathcal{R}} \stackrel{j_*}{\rightarrow} H_n(M,A)^{\mathcal{R}}\stackrel{\partial}{\rightarrow} H_{n-1}(A)^{\mathcal{R}} \rightarrow \cdots,$$
where $\partial:H_n(M,A)^{\mathcal{R}}\rightarrow H_{n-1}(A)^{\mathcal{R}}$ is the natural boundary map, i.e, $\partial [\alpha]$ is the class of the cycle $\partial \alpha$ in $H_{n-1}(A)^{\mathcal{R}}$.
\end{prop}
\begin{proof}It is easy to check that for every $Y\in \mathcal{D}_M$ the following diagram commutes 
\begin{scriptsize}\begin{displaymath}\xymatrix{\cdots H_n(A\cap Y) \ar[r]^{(i_{_Y})_*} \ar[d] &  H_n(Y) \ar[r]^{(j_{_Y})_*} \ar[d] & H_n(Y,A\cap Y) \ar[r]^{\partial} \ar[d] & H_{n-1}(A\cap Y) \ar[r]^{(i_{_Y})_*} \ar[d] & H_{n-1}(Y) \ar[d]\cdots \\
\cdots H_n(A) \ar[r]^{i_*} &  H_n(M) \ar[r]^{j_*} & H_n(M,A) \ar[r]^{\partial}  & H_{n-1}( A)  \ar[r]^{i_*}  & H_{n-1}(M)\cdots}\end{displaymath}\end{scriptsize} 
\hspace*{-0.1cm}where $i_{_Y}:(A\cap Y,\emptyset)\rightarrow (Y,\emptyset)$ and $j_{_Y}:(Y,\emptyset)\rightarrow (Y,A\cap Y)$ are the inclusions (and the superscript $\mathcal{R}$ has been omitted). By the o-minimal exactness axiom the first sequence is exact for every $Y\in \mathcal{D}_M$. Hence, if we take the direct limit, the sequence remains exact. The result then follows from Theorem \ref{teo:limdirhom}.
\end{proof}
\begin{prop}[Excision axiom] Let $M$ be an \ld \ and let $A$ be an admissible subspace of $X$. Let $U$ be an admissible open subspace of $M$ such that $\overline{U}\subset int(A)$. Then the inclusion $j:(M-U,A-U)\rightarrow (M,A)$ induces an isomorphism $j_*:H_n(M-U,A-U)^{\mathcal{R}}\rightarrow H_n(M,A)^{\mathcal{R}}$. 
\end{prop}
\begin{proof}By Theorem \ref{teo:limdirhom}.(ii), it is enough to prove that for each definable subspace $Y$ of $M$ the inclusion $j_{_Y}:(Y-U_Y,A_Y-U_Y)\rightarrow (Y,A_Y)$ induces an isomorphism in homology, where $U_Y=U\cap Y$ and $A_Y=A\cap Y$. So let $Y$ be a definable subspace of $M$. Since $M$ is regular then we can regard $Y$ as an definable set. Now, $cl_{_{Y}}(U_{Y})\subset \overline{U\cap Y}\cap Y\subset \overline{U}\cap Y \subset int(A)\cap Y \subset int_{_{Y}}(A_Y)$. Finally, by the o-minimal excision axiom, $j_{_{Y}}$ induces an isomorphism in homology.
\end{proof}
The proof of the dimension axiom is trivial.
\begin{prop}[Dimension axiom] If $M$ is a one point set, then \linebreak $H_n(M)^{\mathcal{R}}=0$ for all $n>0$.
\end{prop}

Once we have a well-defined homology functor in the locally definable category, we now see that this functor has a good behavior with respect to model theoretic operators. The following result will be used in Section \ref{shomot} in the proof of the Hurewicz theorems for \ld s.
\begin{teo}\label{teo:invhomolld}The homology groups of \ld s are invariant under elementary extension and o-minimal expansions. 
\end{teo}
\begin{proof}We prove the invariance by o-minimal expansions. So let $\mathcal{R}'$ be an o-minimal expansion of $\mathcal{R}$ and let $M$ be an \ld \ in $\mathcal{R}$. Denote by $\mathcal{D}_M$ the collection of all definable subspaces of $M$. Recall that since $M$ is regular each $Y\in \mathcal{D}_M$ can be regarded as an affine definable space (see Remark \ref{obs:regularity}). Now, since the o-minimal homology groups are invariant under o-minimal expansions (see \cite[Prop.3.2]{02BeO}), for each $Y\in \mathcal{D}_M$ there is a natural isomorphism $F_{_Y}: H_n(Y)^{\mathcal{R}}\rightarrow H_n(Y)^{\mathcal{R}'}$. Hence, there exist a natural isomorphism $F:\underrightarrow{lim}_{_{Y\in \mathcal{D}_M}}H_n(Y)^{\mathcal{R}}\rightarrow \underrightarrow{lim}_{_{Y\in \mathcal{D}_M}}H_n(Y)^{\mathcal{R}'}$. By Theorem \ref{teo:limdirhom} and Remark \ref{obs:directlimit}, we have natural isomorphisms $\psi_1:\underrightarrow{lim}_{_{Y\in \mathcal{D}_M}}H_n(Y)^{\mathcal{R}}\rightarrow H_n(M)^{\mathcal{R}}$ and $\psi_2:\underrightarrow{lim}_{_{Y\in \mathcal{D}_M}}H_n(Y)^{\mathcal{R}'}\rightarrow H_n(M)^{\mathcal{R}'}$. Finally, we consider the natural isomorphism $\psi_2 \circ F \circ \psi_1^{-1}: H_n(M)^{\mathcal{R}}\rightarrow H_n(M)^{\mathcal{R}'}$. The proof of the invariance by elementary extensions is similar.\end{proof}
\begin{nota}\label{nota:homolsaomin}We will denote by $\theta$ the natural isomorphism  given by Theorem \ref{teo:invhomolld} between the semialgebraic and the o-minimal homology groups of a regular and paracompact locally semialgebraic space. Note that when we restrict the above $\theta$ to the definable category we obtain the natural isomorphism of \cite[Prop.3.2]{02BeO}.
\end{nota}
\section{Examples of locally definable spaces}\label{sexamples}
We begin this section discussing some natural examples of subsets of $R^n$ carrying a special \LD \ structure. In the second subsection we will consider $\bigvee$-groups as \LD s. Another important class of examples will be shown in Section \ref{shomotgrld}, where we prove the existence of covering maps for \ld s.

\subsection{Subsets of $R^n$ as ld-spaces}\label{sexRn}
\begin{ejemcur}\label{ex}Fix an $n\in \mathbb{N}$ and a collection $\{M_i\}_{i\in I}$ of definable subsets of $R^n$ such that $M_i\cap M_j$ is open in both $M_i$ and $M_j$ (with the topology they inherit from $R^n$) for all $i,j\in I$. Then, clearly $(M_i,id|_{M_i})_{i\in I}$ is an atlas for $M:=\bigcup_{i\in I}M_i$ and hence $M$ is an \LD.
\end{ejemcur}
Let $M\subset R^n$ be an  \LD \ as in Example \ref{ex}. Then it is easy to prove that a definable subspace of $M$ is a definable subset of $R^n$. However, consider the particular example where $M_i:=(-i,i)\subset R$ for $i\in \mathbb{N}$, so that $M=\bigcup_{i\in \mathbb{N}}M_i=Fin(R)$. Note that if $R=\mathbb{R}$ then $\mathbb{R}$ is not a definable subspace of $Fin(\mathbb{R})$ ($=\mathbb{R}$). This also shows that the structures of $\mathbb{R}$ as \LD \ and definable set are different. The latter example can be used also to show that the image of an admissible subspace of an \LD \ by an ld-map might not be admissible. For, take $R$ a non-archemedian real closed field and the ld-map $id:Fin(R)\rightarrow R:x\mapsto x$. Clearly, $Fin(R)$ is not an admissible subspace of $R$ since the admissible subspaces of $R$ are exactly the definable ones.

Nevertheless, we point out that if $M\subset R^n$ is as in Example \ref{ex} with each $M_i$ defined over $A$, $A\subset R$, $|A|<\kappa$, and $\mathcal{R}$ is $\kappa$-saturated, then a definable subset of $R^n$ contained in $M$ is a definable subspace of $M$. For, if $X\subset M$ is a definable subset, to prove that it is a definable subspace it suffices to show that it is contained in a finite union of charts $M_i$, which is clear by saturation.

In general, the topology of an \LD \ $M\subset R^n$ as in Example \ref{ex} does not coincide with the topology it inherits from $R^n$. Consider the following example in $\mathbb{R}$. Take $M_0:=\{ 0\}$ and $M_i:=\{\frac{1}{i}\}$ for $i\in \mathbb{N}\setminus \{0\}$. $M_0$ is open in the topology of $M$ as \LD \ but it is non-open with the topology that $M$ inherits from $\mathbb{R}$. It is well known that this also happen at the definable space level (see Robson's example of a non-regular semialgebraic space --Chapter 10 in \cite{98Dr}--). Moreover, Robson's example shows that even in the presence of saturation the topologies might not coincide.

Finally, let $M\subset R^n$ is as in Example \ref{ex} with each $M_i$ defined over $A$, $A\subset R$, $|A|<\kappa$. Furthermore, assume that $\mathcal{R}$ is $\kappa$-saturated and that the topology of $M$ as \LD \ coincide with the topology it inherits from $R^n$. Then let us note that in this case a definable subspace of $M$ (which as we have seen is also a definable subset of $R^n$) is definably connected if and only if it is connected.

Next, we  show that an \LD \ $M$ as in Example \ref{ex} might not be paracompact.
\begin{ejemcur}\label{expara} Let $M$ be as in Example \ref{ex} with $M_i=\{(x,y)\in R^2:y<0\}\cup \{(x,y)\in R^2:x=i\}$ for each $i\in \mathbb{N}$. The set $X=\{(x,y)\in R^2:y<0\}$ is a definable subspace of $M=\bigcup_{i\in \mathbb{N}} M_i\subset R^2$. However, $\overline{X}=X\cup \{(i,0)\in R^2: i\in \mathbb{N} \}$ is not a definable subspace of $M$. In particular, $M$ is not paracompact (see Fact \ref{fact:claudef}.(1)).
\end{ejemcur}

We finish by showing that another class of subsets that classically has been considered as ``locally semialgebraic subsets'' (for example, by S. Lojasiewicz) can be treated inside the theory of \LD s.
\begin{ejemcur}\label{ejLoj}Let $M$ be a subset of $R^n$ such that for every $x\in M$ there is an open definable neighbourhood $U_x$ of $x$ in $R^n$ with $U_x\cap M$ definable subset. Let $M_x:=U_x\cap M$ for each $x\in M$.  Then $M$ is an \LD \ with the atlas $(M_x,id|_{M_x})_{x\in M}$.
\end{ejemcur}
Using the notation of Example \ref{ejLoj}, it is clear that $M_x\cap M_y$ is definable and open in both $M_x$ and $M_y$ for all $x,y\in M$ and therefore $M$ is an \LD \ as in Example \ref{ex}. Moreover, the topology of $M$ as \LD \ equals the one it inherits from $R^n$.
\subsection{$\bigvee$-definable groups}\label{ssexVgr}In this section we will assume $\mathcal{R}$ is $\aleph_1$-saturated. The $\bigvee$-definable groups have been considered by several authors as a tool for the study of definable groups in o-minimal structures. Y. Peterzil and S. Starchenko give the following definition  in \cite{00PS}. A group $(G,\cdot)$ is a $\bigvee$-\textit{definable group} over $A$, $A\subset R$, $|A|<\aleph_1$, if there is a collection $\{X_i:i\in I\}$ of definable subsets of $R^n$ over $A$ such that $G=\bigcup_{i\in I}X_i$ and for every $i,j\in I$ there is $k\in I$ such that $X_i\cup X_j\subset X_k$ and the restriction of the group multiplication to $X_i\times X_j$ is a (not necessarily continuous) definable map into $R^n$. M. Edmundo introduces in \cite{06E} a notion of restricted $\bigvee$-definable group which he calls ``locally definable'' group. Our purpose in this section is to include both notions within the theory of \LD s.

In \cite{00PS}, some (topological) topics of $\bigvee$-definable groups are discussed to study the definable homomorphisms of abelian groups in o-minimal structures and, in particular, they prove the following result.
\begin{fact}\emph{\cite[Prop. 2.2]{00PS}}\label{prop:peter} Let $G\subset R^n$ be a $\bigvee$-definable group. Then, there is a uniformly
definable family $\{V_a : a \in S\}$ of subsets of $G$ containing the identity element $e$ and
a topology $\tau$ on $G$ such that $\{V_a : a \in S\}$ is a basis for the $\tau$-open neighbourhoods
of $e$ and $G$ is a topological group. Moreover, every generic $h \in G$ has an open
neighbourhood $U\subset N^n$ such that $U \cap G$ is $\tau$-open and the topology which $U \cap G$
inherits from $\tau$ agrees with the topology it inherits from $R$, and the topology $\tau$ is the unique one with the above properties.
\end{fact}
Because of the above fact is natural to introduce the following concept.
\begin{deff}\label{dldgr}We say that a group $(G,\cdot)$ is an \textbf{ld-group} if $G$ is an \LD \ and both $\cdot:G\times G \rightarrow G$ and $^{-1}:G\rightarrow G$ are ld-maps. If $G$ is moreover paracompact as \LD \ we say that $G$ is an \textbf{LD-group} (note that since every ld-group is a topological group it is regular).
\end{deff}
We will see that every $\bigvee$-definable group (with its group topology) is an ld-group. We begin with the following result.
\begin{lema}\label{lema:entorno}Let $G\subset R^n$ a $\bigvee$-definable group over $A$ and let $\tau$ be the topology of Fact \ref{prop:peter}. Then, for every generic $g \in G$ there is a definable \begin{footnotesize}\emph{OVER}\end{footnotesize} $A$ subset $U_g\subset G$  which is $\tau$-open and such that the topology which $U_g$ inherits from $\tau$ agrees with the topology it inherits from $R^n$.
\end{lema}
\begin{proof}By Fact \ref{prop:peter} it suffices to prove that the parameter set $A$ is preserved. Write $G=\bigcup_{i\in I}X_i$. The dimension of $G$ is defined as $max\{dim(X_i):i\in I\}$. Fix an $X_i$ of maximal dimension and a generic $g\in X_i$. We can assume that $X_i^{-1}=X_i$. Let $X_j$ be such that $X_iX_iX_i\subset X_j$. All the definable sets we shall consider in the proof are definable subsets of $X_j$. For each $a\in X_i$ we consider the definable set\\

\vspace{-0.3cm}
$W_a=\{x\in X_i: \forall \delta>0 \exists \epsilon>0 \ B(x,\epsilon)\subset xa^{-1}B(a,\delta)\wedge$\\

\vspace{-0.3cm}
\hspace*{5cm} $\forall \epsilon>0 \exists \delta>0 \ xa^{-1}B(a,\delta)\subset B(x,\epsilon) \},$\\ 

\vspace{-0.3cm}
\hspace*{-0.6cm}where $B(x,\epsilon)=\{y\in X_i:|y-x|<\epsilon \}$. We also consider the definable set $$V=\{y\in X_i: W_y \textrm{ is large in } X_i \}.$$ 

\hspace{-0.6cm}By Claim 2.3 of \cite[Prop. 2.2]{00PS}, for every $h\in X_i$ generic over $A$ and $g$ we have that  $h\in W_g$ and therefore $g\in V$. Moreover, since $g$ is generic, we have that $g\in U:=int_{_{X_i}}(V)$ (the interior with respect to the topology of the ambient space $R^n$), which is a definable over $A$ subset of $X_i$. Fix $a\in U$. We shall prove that\\

\vspace{-0.3cm}
\hspace*{-0.6cm}(i) for every $\epsilon>0$ there is $\delta>0$ such that  $ag^{-1}B(g,\delta)\subset B(a,\epsilon)$, and\\
(ii) for every $\epsilon>0$  there is $\delta>0$ such that  $ga^{-1}B(a,\delta)\subset B(g,\epsilon)$.

\vspace{0.2cm}
\hspace{-0.6cm}Granted (i) and (ii), note that $U_g:=U$ is the desired neighbourhood of $g$. Let us show (i). Consider a generic $h\in X_i$ over $A,a$.
Since $h\in W_a$, there is $\widetilde{\delta}>0$ such that $ah^{-1}B(h,\widetilde{\delta})\subset B(a,\epsilon)$. By Claim 2.3 of \cite[Prop. 2.2]{00PS}, there is $\delta>0$ such that $g^{-1}B(g,\delta)\subset h^{-1}B(h,\widetilde{\delta})$. Hence $ag^{-1}B(g,\delta)\subset ah^{-1}B(h,\widetilde{\delta})\subset B(a,\epsilon)$. The proof of (ii) is similar.
\end{proof}
The following technical fact can be easily deduced from the proof of \cite[Prop 2.11]{06E}.
\begin{fact}\label{fedmundo}Let $G=\bigcup_{i\in I}X_i$ be an $\bigvee$-definable group over $A$. Let $V=\bigcup_{k\in \Lambda}V_k$ (directed union) be a subset of $G$ such that each $V_k$ is definable over $A$ and $V$ is large in $G$, i.e, every generic point of $G$ is contained in $V$. Then there is a collection of elements $\{b_j\in G:j\in J\}$ with each $b_j$ definable over $A$, such that each $X_i$ is contained in a finite union of subsets of the form $b_jV_k$. In particular, $G=\bigcup_{j\in J}b_jV$.
\end{fact}
As it was pointed out by Y. Peterzil to us, a stronger version of the above fact can be proved. In particular, and using the notation of Fact \ref{fedmundo}, there exist $b_1,\ldots,b_n\in G$, $n=dim(G)$, such that  $G=\bigcup_{i=1}^{n}b_nV$ (it is enough to adapt the proof of \cite[Fact. 4.2]{07P}). However, in this case we do not know if $b_1,\ldots,b_n$ are definable over $A$. Since we are interested in preserving the parameter set we will use the above Fact \ref{fedmundo}.
\begin{teo}\label{Vgrldgr}Let $G\subset R^n$ be a $\bigvee$-definable group over $A$. Let $A\subset C \subset R$. Then\\
\emph{(i)} $G$ with its group topology (from Fact \ref{prop:peter}) is an ld-group over $A$,\\
\emph{(ii)} a subset $X$ of $G$ is a definable subset of $R^n$ over $C$ if and only if it is a definable subspace of $G$ over $C$, and\\
\emph{(iii)} given a definable subspace $X$ of $G$ over $C$ , its closure $\overline{X}$ (with respect to the group topology) is a definable subspace of $G$ over $C$.
\end{teo}
\begin{proof}(i) Let $\mathcal{G}$ be the collection of all generics points of $G$. For each $g\in \mathcal{G}$, let $U_g$ be the definable over $A$ subset of $G$ of Lemma \ref{lema:entorno}. Consider the subset $V=\bigcup_{g\in \mathcal{G}}U_g$ of $G$, which is large in $G$. By Fact \ref{fedmundo}, there is a collection  $\{b_j\in G:j\in J\}$, with each $b_j$ definable over $A$, such that $G=\bigcup_{j\in J}b_jV$. For each $j\in J$ and $g\in \mathcal{G}$, consider the definable set $V_{j,g}:=b_jU_g$ and the bijection $\psi_{j,g}:V_{j,g}\rightarrow U_g: y \mapsto b_j^{-1}y$. Finally, it is easy to check that $\{(V_{j,g},\psi_{j,g})\}_{j\in J,g\in \mathcal{G}}$ is an atlas of $G$ and therefore $G$ is an ld-group over A.\\
(ii) It is clear that if $X\subset G$ is a definable subspace over $C$ then it is a definable subset of $R^n$ over $C$. So, let $X$ be a definable subset of $R^n$ over $C$ and consider the atlas $\{(V_{j,g},\psi_{j,g})\}_{j\in J,g\in \mathcal{G}}$ of $G$ constructed in the proof of (i). Since $X$ is definable over $C$ we have that $\psi_{j,g}(X\cap V_{j,g})=b_j^{-1}X\cap U_g$ is also definable over $C$ for every $j\in J$ and $g\in \mathcal{G}$. Hence, it is enough to show that $X$ is contained in a finite union of the sets $V_{j,g}$ (which are defined over $A$) and this is clear by saturation since they cover $G$.\\
(iii) Let $X$ be a definable subspace of $G$ over $C$ and write $G=\bigcup_{i\in I}X_i$. By (ii) $X$ is a definable subset of $R^n$ over $C$. We will show that $\overline{X}$ is a definable subset of $R^n$ over $C$ (this is enough also by (ii)). Fix a generic point $g$ of $G$ and let $U_g$ as in Lemma \ref{lema:entorno}. Firstly, let us show that $\overline{X}\subset X_j$ for some $j\in I$. Since $\{X_i\}_{i\in I}$ is a directed family and $X$ and $U_g$ are definable, there is $j\in I$ such that $XU_{g}^{-1}g\subset X_j$. Now, if $y\in \overline{X}$ then $yg^{-1}U_g\cap X\neq \emptyset$ and hence $y\in XU_g^{-1}g\subset X_j$. Finally, $\overline{X}=\{y\in X_j: g\in cl_{_{U_g}}(gy^{-1}X\cap U_g)\}$ is clearly a definable subset of $R^n$ over $C$, where $cl_{_{U_g}}(-)$ denotes the closure in $U_g$ with respect to the inherited topology from the ambient space $R^n$.
\end{proof}
Theorem \ref{Vgrldgr}.(iii) states that in a $\bigvee$-group we have a good relation between both the topological and the definable setting as it happens with LD-spaces (see Fact \ref{fact:claudef}.(1)). However, not every $\bigvee$-definable group is paracompact (or Lindel\"of) as ld-group. To see this, take an $\aleph_1$-saturated elementary extension $\mathcal{R}$ of the o-minimal structure $\left\langle \mathbb{R},<,+,-,\cdot,c\right\rangle_{c\in \mathbb{R}}$. Firstly, consider the collection $\mathcal{F}$ of finite subsets of $\mathbb{R}$. Then $(G,+)$, where $G=\bigcup_{F\in \mathcal{F}}F\subset R$ and $+$ is the usual addition, is a $\bigvee$-definable group over $\emptyset$ which is not Lindel\"of as ld-group. Note that the group topology of $G$ as $\bigvee$-definable group is the discrete one. Secondly, consider $(G,+)$, where $G=\bigcup_{r\in \mathbb{R}}(-r,r)\subset R$ and $+$ is the usual addition. The group $(G,+)$ is a $\bigvee$-definable group which is not Lindel\"of as ld-group. Since it is connected, $(G,+)$ is not paracompact (see Fact \ref{fact:claudef}.(2)).

In \cite{06E}, M. Edmundo considers $\bigvee$-definable groups $G=\bigcup_{i\in I}X_i$ over $A$ with the restriction $|I|<\aleph_1$ (which already implies the restriction $|A|<\aleph_1$), he calls them \textit{``locally definable''} groups. This restriction on the cardinality of $I$ allows Edmundo to prove results using techniques which are not available in the general setting of $\bigvee$-definable groups. As he notes the main examples of $\bigvee$-definable groups are of this form: the subgroup of a definable group generated by a definable subset and the coverings of definable groups. The restriction on the cardinality of $|I|$ of the ``locally definable'' groups has also the following consequences on them as \LD s.
\begin{teo}\label{teo:ldgparacom}\emph{(i)} Every ``locally definable'' group over $A$ with its group topology is a Lindel\"of LD-group over $A$.\\
\emph{(ii)} Moreover, every  Lindel\"of LD-group over $A$ is ld-isomorphic to a ``locally definable'' group over $A$ (considered as an LD-group by (i)).
\end{teo}
\begin{proof}(i) Let $G$ be a ``locally definable'' group over $A$. By Theorem \ref{Vgrldgr}.(i), $G$ is an ld-group over $A$. We first show that $G$ is  Lindel\"of. Recall the notation of Theorem \ref{Vgrldgr}.(i). Write $G=\bigcup_{i\in I}X_i$, with $|I|<\aleph_1$. Since $I$ is countable, to prove that $G$ is Lindel\"of we can assume that the language is countable (recall that Lindel\"of property is invariant under o-minimal expansions by Proposition \ref{prop:epanext}). Now, since for each generic $g\in G$ the definable subset $U_g$ of Lemma \ref{lema:entorno} is definable over $A$, the collection $\{U_g:g\in G \textrm{ generic}\}$ is countable. Hence, the atlas $\{(V_{j,g},\psi_{j,g})\}_{j\in J,g\in \mathcal{G}}$ of the proof of Theorem \ref{Vgrldgr}.(i) is also countable and so $G$ is Lindel\"of. Having proved the latter, the paracompacity follows from  Theorem \ref{Vgrldgr}.(iii) and Fact \ref{fact:claudef}.\\
(ii) Let $G$ be a Lindel\"of LD-group over $A$. Since $G$ is regular and paracompact, by Fact \ref{fact:triangulation} and Remark \ref{rmk:paratriang} there is a ld-triangulation $\psi:|K|\rightarrow G$ over $A$. Now, since $G$ is a group, the dimension of $K$ is finite. Furthermore, since $G$ is Lindel\"of, the admissible covering $\{St_{|K|}(\sigma):\sigma\in K\}$ of $|K|$ has a countable subcovering of $|K|$. From this fact we deduce that $K$ is countable. Then, by \cite[Prop. II.3.3]{85DK}, we can assume that the realization $|K|$ lie in $R^{2n+1}$, $n=dim(K)$, and that the topology it inherits from $R^{2n+1}$ coincide with its topology as \ld. Now, define in $|K|$ a group operation via the ld-isomorphism $\psi$ and the group operation of $G$. With this group operation, $|K|$ is an LD-group which we will denote by $H$. Of course, $G$ is ld-isomorphic to $H$ via $\psi$. On the other hand, we can consider $|K|$ as a ``locally definable'' group. For, let $\mathcal{F}$ the collection of all finite simplicial subcomplexes of $K$. Clearly, $|K|=\bigcup_{L\in \mathcal{F}}|L|$ with the group operation obtained via $\phi$ is a ``locally definable'' group over $A$. Indeed, since the group operation is an ld-map, its restriction to $|L_1|\times |L_2|$ is a definable map into $R^{2n+1}$ for all $L_1,L_2\in \mathcal{F}$. Finally, since the group operation is already continuous and the topology of $|K|$ as \LD \ coincide with the one inherited form $R^{2n+1}$, the ``locally definable'' group $|K|$ with the ld-group structure obtained in part (i) is exactly $H$.
\end{proof}
\begin{cor}Let $G$ be a ``locally definable'' group over $A$. Then, there is a ld-triangulation $\psi:|K|\rightarrow G$ of $G$ over $A$ with $|K|\subset R^{2n+1}$, $n=dim(G)$, and such that the topology of $|K|$ as \ld \ coincide with the one inherited from $R^{2n+1}$. Moreover, $|K|$ with the group operation inherited from $G$ via $\psi$ is also a ``locally definable'' group over $A$ whose group topology equals the one inherited from $R^{2n+1}$.
\end{cor}

Let us point out that there are important examples of $\bigvee$-groups which are not Lindel\"of \ld s (and hence not ``locally definable'' groups). The group of definable homomorphism between abelian groups were used in \cite{00PS} as a tool to study interpretability problems. In particular, given to abelian definable groups $A$ and $B$ over $C$, $C\subset R$, it is proved there that the group of definable homomorphisms $\mathcal{H}(A,B)$ from $A$ to $B$ is a $\bigvee$-definable group over $C$ (see \cite[Prop. 2.20]{00PS}). Note that $\mathcal{H}(A,B)$ might not be a ``locally definable'' group (see the Examples at the end of Section 3 in \cite{00PS}). Nevertheless, $\mathcal{H}(A,B)$ is an LD-group. Indeed, we have already seen in Theorem \ref{Vgrldgr}.(i) that it is an ld-group (and hence regular). To prove paracompactness, consider its connected component $\mathcal{H}(A,B)^{0}$, which is a definable group by \cite[Thm. 3.6]{00PS}. Then, by Theorem \ref{Vgrldgr}.(ii), $\mathcal{H}(A,B)^{0}$ is a definable subspace of $\mathcal{H}(A,B)$. Hence, $\{g\mathcal{H}(A,B)^{0}:g\in \mathcal{H}(A,B) \}$ is a locally finite covering of $\mathcal{H}(A,B)$ by open definable subspaces and therefore $\mathcal{H}(A,B)$ is paracompact. As we will see in the next section, the notion of connectedness used in  \cite{00PS} for $\bigvee$-groups differs from the one used here. However, in this particular case, since $\mathcal{H}(A,B)^{0}$  is definable, both notions coincide.
\section{Connectedness}\label{sec:connectedness}
Recall that an \LD \ $M$ is connected if there is no admissible nonempty proper clopen subspace $U$ of $M$. We can also extend the natural concept of ``path connected'' for definable spaces to the locally definable ones. Specifically, we say that an admissible subspace $X$ of an \LD \ $M$ is \textit{\textbf{path connected}} if for every $x_0,x_1\in X$ there is a ld-path $\alpha:[0,1]\rightarrow X$ such that $\alpha(0)=x_0$ and $\alpha(1)=x_1$. Naturally, the (path) \textit{\textbf{connected components}} of an ld-space are the maximal (path) connected subsets.
\begin{fact}\emph{\cite[Prop. I.3.18]{85DK}}\label{fact:conexoporarcos} Every path connected component of an \LD \ is a clopen admissible subspace.
\end{fact}
From the above fact we deduce that the connected and path-connected components of an \LD \ are admissible subspaces and coincide. In particular, every connected ld-space is path connected (the converse is trivial).

Note that given two \LD s $M$ and $N$, with $M$ connected and $N$ discrete, every ld-map $f:M\rightarrow N$ is constant. For, since $N$ is discrete, $\{y\}$ is a clopen definable subspace of $N$ for all $y\in N$. Therefore the admissible subspace $f^{-1}(y)$ of $M$ is clopen for all $y\in N$. Since $M$ is connected, $M=f^{-1}(y_0)$ for some $y_0\in N$.

Since $\bigvee$-definable groups were first considered several non equivalent notions of connectedness have been used. As we will see here some of them are not really adequate and lead to pathological examples. Fix a $\bigvee$-definable group $G=\bigcup_{i\in I}X_i\subset R^n$ over $A$, $A\subset R$, $|A|<\aleph_1$, in an $\aleph_1$-saturated o-minimal expansion $\mathcal{R}$ of a real closed field $R$. Here, we say that $G$ is connected if it is so as ld-group (see Theorem \ref{Vgrldgr}). In \cite{00PS}, $G$ is said to be $\mathcal{M}$-connected (\textit{PS-connected}, for us) if there is no definable set $U$ in $R^n$ such that $U\cap G$ is a nonempty proper clopen subset with the group topology of $G$. In \cite{06E}, $G$ is said to be connected (\textit{E-connected}, for us) if there is no definable set $U\subset G$ such that $U$ is a nonempty proper clopen subset with the group topology of $G$. Finally, in \cite{08OP}, $G$ is said to be connected (\textit{OP-connected}, for us) if all the $X_i$ can be chosen to be definably connected with respect to the definable subspace structure it inherits from $G$ as ld-group. Notice that in \cite{08OP} the situation is simpler because $G$ is a subgroup of a definable group and hence embedded in some $R^n$, so each $X_i$ is connected with respect to the ambient $R^n$ (see Section \ref{sexamples}.1).

For $\bigvee$-definable groups the relation of the above notions is as follows:
\begin{center} 
OP-connected $\Leftrightarrow$ Connected $\Rightarrow$ PS-connected $\Rightarrow$ E-connected.
\end{center}
The second and third implications are clear by definition. Furthermore, the following examples show that these implications are strict.
\begin{ejemcur}Let $R$ be a non archimedean real closed field. Consider the definable set $B=\{(t,-t)\in R^2:t\in [0,1] \}\ \cup \{(t,t-2)\in R^2:t\in [1,2] \}$. For each $n\in \mathbb{N}$, consider the definable set $X_n=(\bigcup_{i=-n}^n (2i,0)+B) \cup (\bigcup_{i=-n}^n (2i,-\frac{1}{2})+B)\subset R^2$. Define a group operation on $G=\bigcup_{n\in \mathbb{N}} X_n$ via the natural bijection of $G$ with Fin($R$)$\times \mathbb{Z}/2\mathbb{Z}$, where Fin($R$)$=\{x\in R: |x|<n \textrm{ for some } n\in \mathbb{N} \}$. Then, $G$ with this group operation is a $\bigvee$-definable group.
\end{ejemcur}
Note that the topology of $G$ inherited from $R^2$ coincide with its group topology. $G$ is not connected as an \LD \ because it has two connected components. However, $G$ is PS-connected because any definable subset of $R^2$ which contains one of these connected components must have a nonempty intersection with the other component.
\begin{ejemcur}\cite{08BE} Let $R$ be a non archimedean real closed field and consider the definable sets $X_n=(-n,-\frac{1}{n})\cup (\frac{1}{n},n)$ for $n\in \mathbb{N}$, $n>1$. Then, $G=\bigcup_{n>1}X_n$ is a $\bigvee$-definable group with the multiplicative operation of $R$.
\end{ejemcur}

Here, again, the topology $G$ inherits from $R^2$ coincide with its group topology. The $\bigvee$-definable group $G$ is not PS-connected since it is the disjoint union of the clopen subsets $\{x\in R:x>0\}\cap G$ and $\{x\in R:x<0\}\cap G$. But neither of these subsets is definable and therefore $G$ is E-connected.

Note that in both examples we can define in an obvious way an ld-map $f:G\rightarrow \mathbb{Z}/2\mathbb{Z}$ which is not constant and therefore the remark we made at the beginning of this section is not true if we replace connectedness by PS-connectedness or E-connectedness.

Even though there are pathological examples, the results in \cite{00PS} are correct for PS-connectedness. For the results in \cite{06E}, one should substitute E-connectedness by connectedness (see \cite{08BE}).

We now prove the equivalence between both OP-connectedness and connectedness.
\begin{prop}\label{prop:OPconnected}Let $G$ be a $\bigvee$-definable group over $A$. Then, $G$ is OP-connected if and only if $G$ is connected.
\end{prop}
\begin{proof}Firstly, recall that by Theorem \ref{Vgrldgr} a subset of $G$ is a definable subspace if and only if it is a definable subset of $R^n$. Let $G$ be an OP-connected $\bigvee$-definable group, i.e, such that $G=\bigcup_{i\in I}X_i$ with $X_i$ definably connected for all $i\in I$. Consider a nonempty admissible clopen subspace $U$ of $G$. Since $U$ is not empty and each $X_i$ is definably connected, there is $i_0\in I$ such that $X_{i_0}\subset U$. Now, for every $i\in I$ there is $j\in I$ with $X_{i_0}\cup X_i\subset X_j$. Since $X_j$ is definably connected and $\emptyset \neq X_{i_0}\subset X_j\cap U$ we have that $X_j\subset U$ and, in particular, $X_i\subset U$. So we have proved that for every $i\in I$, $X_i\subset U$. Hence $U=G$, as required.

Now, let $G$ be a connected $\bigvee$-definable group over $A$. Let $\mathcal{C}$ be the collection of all connected definable subspaces over $A$ of $G$ which are connected and contain the unit element of $G$. It is enough to show that $G=\bigcup_{X\in \mathcal{C}}X$. Note that we just consider the connected definable subspaces of $G$ which are definable over $A$ because we need to preserve the parameter set. So let $x\in G$. By Fact \ref{fact:conexoporarcos}, $G$ is also path connected and hence there is an ld-curve $\alpha:I\rightarrow G$ such that $\alpha(0)=x$ and $\alpha(1)=e$. Since $\alpha(I)$ is definable and $G$ is an ld-group over $A$, a finite union of charts (which are definable over $A$) contains $\alpha(I)$. Hence $\alpha(I)$ is contained in a definable over $A$ subset $X$ of $G$. Taking the adequate connected component, we can assume that $X$ is connected. Hence $x\in X\in \mathcal{C}$.
\end{proof}
\begin{cor}A $\bigvee$-definable group is OP-connected if and only if is path-connected. 
\end{cor}
\begin{proof}By Fact \ref{fact:conexoporarcos} and Proposition \ref{prop:OPconnected}. 
\end{proof}
\section{Homotopy theory in \ld s}\label{shomot}
Once we have defined the category of locally definable spaces, in the following section we will develop a homotopy theory for \ld s, that is, regular and paracompact locally definable spaces. This section is divided in Subsections \ref{shomotsetLD}, \ref{shomotgrld} and \ref{shurewiczld}, which are the locally definable analogues of Sections 3,4 and 5 of \cite{07pBO}, respectively.
\subsection{Homotopy sets of locally definable spaces}\label{shomotsetLD}
The homotopy sets in the locally definable category are defined as in the definable one just substituting the definable maps by the locally definable ones (see Section 3 in \cite{07pBO}). Specifically, let $(M,A)$ and $(N,B)$ be two pairs of \ld s, i.e., $M$ and $N$ are \ld s and $A$ and $B$ are admissible subspaces of $M$ and $N$ respectively. Let $C$ be a closed admissible subspace of $M$ and let $h:C\rightarrow N$ be an ld-map such that $h(A\cap
C )\subset B$. We say that two ld-maps $f,g:(M,A)\rightarrow
(N,B)$ with $f|_{C}=g|_{C}=h$, are \textit{\textbf{ld-homotopic relative to $h$}}, denoted by $f \thicksim_h g$, if there exists an ld-homotopy $H:(M\times I,A \times I)\rightarrow (N,B)$ such that $H(x,0)=f(x)$, $H(x,1)=g(x)$ for all $x\in M$ and $H(x,t)=h(x)$ for all $x\in C$ and $t\in I$. The \textit{\textbf{homotopy set of $(M,A)$ and $(N,B)$ relative to $h$}} is the set
$$[(M,A),(N,B)]_{h}^{\mathcal{R}}=\{f: f:(M,A)\rightarrow (N,B) \textrm{ ld-map in } \mathcal{R}, f|_{C}=h \}/\thicksim_h.$$
If $C=\emptyset$ we omit all references to $h$. We shall denote by $\mathcal{R}_0$ the field structure of the real closed field $R$ of our o-minimal structure $\mathcal{R}$. Given two pairs of regular paracompact locally semialgebraic spaces $(M,A)$ and $(N,B)$ and a locally semialgebraic map $h$ as before, note that we can consider both $[(M,A),(N,B)]_h^{\mathcal{R}_{0}}$ and $[(M,A),(N,B)]_h^{\mathcal{R}}$.

The next theorem is the main result of this section and it establishes a strong relation between the locally definable and the locally semialgebraic homotopy. It is the locally definable analogue of \cite[Cor.3.3]{07pBO}. Recall the behavior of the \LD s under o-minimal expansions in Proposition \ref{prop:epanext}.
\begin{teo}\label{teo:locprinci}Let $(M,A)$ and $(N,B)$ be two pairs of regular paracompact locally semialgebraic spaces. Let $C$ be a closed admissible semialgebraic subspace of $M$ and
$h:C\rightarrow N$ a locally semialgebraic map such that $h(A \cap
C )\subset B$. Suppose $A$ is closed in $M$. Then, the map
$$\begin{array}{rcl}
\rho:[(M,A),(N,B)]_{h}^{\mathcal{R}_{0}} & \rightarrow &
[(M,A),(N,B)]_{h}^{\mathcal{R}} \\

[f] & \mapsto & [f],
\end{array}$$
which sends the locally  semialgebraic homotopic class of a locally semialgebraic map to its locally definable homotopic class, is a bijection.
\end{teo}

An important tool for the proof of the above theorem (and in general, for the study of homotopy properties of \ld s) is the following homotopy extension lemma. Even though the proof for locally semialgebraic spaces (see \cite[Cor.III.1.4]{85DK}) can be adapted to the locally definable setting, we have included here an alternative proof which, in particular, does not make use of the Triangulation Theorem of \ld s (see Fact \ref{fact:triangulation}). Firstly, we prove a technical lemma which establishes a gluing principle of ld-maps by closed definable subsets.
\begin{fact}\emph{\cite[Prop. I.3.16]{85DK}}\label{fpegado} Let $M$ be an \LD \ and $\{C_j:j\in J\}$ be an admissible covering of $M$ by closed definable subspaces. Let $N$ be an \LD \ and $f:M\rightarrow N$ be a map (not necessarily continuous) such that $f|_{C_j}$ is an ld-map for each $j\in J$. Then, $f$ is an ld-map.
\end{fact}
\begin{proof} Let $(M_i,\phi_i)_{i\in I}$ be the atlas of $M$. We have to prove that the conditions of Definition \ref{def:locmap} are satisfied. Firstly, note that since the covering $\{C_j:j\in J\}$ is admissible, for each $i\in I$ there is a finite subset $J_i\subset J$ such that $M_i\subset \bigcup_{j\in J_i}C_j$. Therefore, since $f|_{M_i\cap C_j}$ is continuous and $M_i\cap C_j$ is a closed subset of $M_i $ for all $j\in J_i$, $f|_{M_i}$ is also continuous for every $i\in I$. Now, to prove that $f(M_i)$ is a definable subspace of $N$ for each $i\in I$,  note that, since each $f|_{C_j}$ is an ld-map and $C_j$ is a definable subspace of $M$, $f(M_i\cap C_j) $ is a definable subspace of $N$ for all $i\in I$ and $j\in J$. Hence, $N_i:=f(M_i)=\bigcup_{j\in J_i}f(M_i\cap C_j)$ is a definable subspace of $N$ for each $i\in I$.  Finally, the map $f|_{M_i}:M_i\rightarrow N_i$ is definable since $f|_{M_i\cap C_j}:M_i\cap C_j\rightarrow N_i$ is definable for all $j\in J_i$.\end{proof}
\begin{lema}[\textbf{\textbf{Homotopy extension lemma}}]\label{lema:extlochom}Let $M,N$ be two \ld s and let $A$ be a closed admissible subspace of $M$. Let $f:M\rightarrow N$ be an ld-map and $H:A\times I\rightarrow N$ a ld-homotopy such that $H(x,0)=f(x)$ for all $x\in A$. Then, there exists a ld-homotopy $G:M\times I\rightarrow N$ such that $G(x,0)=f(x)$ for all $x\in M$ and $G|_{A\times I}=H$. 
\end{lema}
\begin{proof}Without loss of generality, we can assume that $M$ is connected and hence, by Fact \ref{fact:claudef}.(2), that $M$ is Lindel\"{o}f. Let $(M_k,\phi_k)_{k\in \mathbb{N}}$ be an atlas of $M$. Consider $X_n:=\bigcup_{k=0}^n\overline{M_k}$ for each $n\in \mathbb{N}$. By Fact \ref{fact:claudef}.(1) each $X_n$ is a closed definable subspace of $M$ and hence $\{X_n:n\in \mathbb{N}\}$ is an admissible covering by closed definable subspaces such that $X_n\subset X_{n+1}$ for all $n\in \mathbb{N}$. Take the restrictions $f_{n}:=f|_{X_n}$ and $H_n:=H|_{A_n\times I}$, where $A_n$ is the closed definable subspace $A \cap X_n$. Moreover, since $M$ is regular, we can regard each $X_n$ as an affine definable space (see Remark \ref{obs:regularity}). Now, by the o-minimal homotopy extension lemma (Lemma 2.1 in \cite{07pBO}) and applying an induction process, we can find a collection of definable homotopies $G_n:X_n\times I\rightarrow N$ such that $G_n(x,0)=f_n(x)$ for all $x\in X_n$, $G_n|_{X_{n-1}\times I}=G_{n-1}$ and $G_n|_{A_{n}\times I}=H_{n}$. Finally, we define the map $G:M\times I\rightarrow N$ such that $G|_{X_n\times I}=G_n$ for every $n\in \mathbb{N}$. By Fact \ref{fpegado}, the map $G$ is locally definable and, by definition, $G|_{A\times I}=H$ and $G(x,0)=f(x)$ for all $x\in M$.
\end{proof}
\begin{proof}[Proof of Theorem \ref{teo:locprinci}]With the above tools at hand we can follow the lines of the proof of \cite[Thm. III.4.2]{85DK}. Here are the details. As in the definable case, it suffices to prove that $\rho$ is surjective when $A=B=\emptyset$. Indeed, we can do here similar reductions than the ones we followed after \cite[Prop. 3.2]{07pBO} just applying the homotopy extension lemma for \ld s (see Lemma \ref{lema:extlochom}) instead of its definable version. Now, we divide the proof in two cases.\\ 
\textit{ Case $M$ is a semialgebraic space:} Since $M$ is regular, we can assume that it is affine (see Remark \ref{obs:regularity}). Let $f:M\rightarrow N$ be an ld-map such that $f|_C=h$. Since $M$ is semialgebraic, $f(M)$ is a definable subspace of the locally semialgebraic space $N$ and therefore it is contained in the union of a finite number of semialgebraic charts. Hence, there is a semialgebraic subspace $N'$ of $N$  such that $f(M)\subset N'$. Now, since $N$ is regular, we can regard $N'$ also as an affine definable space and therefore we can see the map $f:M\rightarrow N'$ as a definable map between semialgebraic sets (see comments after Definition \ref{def:locmap}). By \cite[Cor. 3.3]{07pBO} (which is the definable version of Theorem \ref{teo:locprinci}), there exist a definable homotopy $H':M\times I\rightarrow N'$ such that $H'(x,0)=f(x)$ for all $x\in M$, $H'(x,t)=h(x)$ for all $x\in C$ and $t\in I$ and $H'(-,1):M\rightarrow N'$ is semialgebraic. Hence, it suffices to consider the definable homotopy $H=i \circ H'$ where $i:N'\rightarrow N$ is the inclusion, to get $\rho([H(-,1)])=[f]$.\\
\textit{General Case:}  Let $f:M\rightarrow N$ be an ld-map such that $f|_{C}=h$. We have to show that $f$ is ld-homotopic relative to $h$ to a locally semialgebraic map. Without loss of generality, we can assume that $M$ is connected and hence, by Fact \ref{fact:claudef}.(2), that $M$ is Lindel\"{o}f. Furthermore, by \cite[Thm. I.4.11]{85DK} (which states the shrinking covering property for regular paracompact locally semialgebraic spaces) there is a locally finite covering $\{X_n:n\in \mathbb{N}\}$ of $M$ by closed semialgebraic subspaces. Consider the closed semialgebraic subspace $Y_n:=X_0\cup \cdots \cup X_n$ and the closed admissible subspace $C_n:=Y_n\cup C$ for each $n\in \mathbb{N}$. By the previous case, there exist a definable homotopy $\widetilde{H}_0:Y_0\times I \rightarrow N$ such that $\widetilde{H}_0(x,0)=f(x)$ for all $x\in Y_0$, $\widetilde{H}_0(-,1):Y_0\rightarrow N$ is a locally semialgebraic map and $\widetilde{H}_0(x,t)=h(x)$ for all $x\in C\cap Y_0$ and $t\in I$. Moreover, by Lemma \ref{lema:extlochom}, there exist an ld-homotopy $H_0:M\times I\rightarrow N$ with $H_0(x,0)=f(x)$ for all $x\in M$, $H_0(x,t)=h(x)$ for all $x\in C$ and $t\in I$ and such that $H_0|_{Y_0\times I}=\widetilde{H}_0$. In particular, $g_0:=H_0|_{C_0\times\{1\}}$ is a locally semialgebraic map with $g_0|_{C}=h$. Now, by iteration we obtain a sequence of ld-homotopies $\{H_n:M\times I\rightarrow N:n\in \mathbb{N}\}$ such that

\vspace*{0.2cm}
\hspace{-0.6cm}(i) $g_n:=H_n|_{C_n\times\{1\}}$ is a locally semialgebraic map,\\
(ii) $H_{n+1}(x,t)=g_{n}(x)$ for all $(x,t)\in C_{n}\times I$ (so $g_{n+1}|_{C_{n}}=g_{n}$), and\\
(iii) $H_{n+1}|_{M\times{\{0\}}}=H_{n}|_{M\times{\{1\}}}$.
\vspace{0.2cm}

Note that in particular $H_n(x,t)=g_0(x)=h(x)$ for all $(x,t)\in C\times I$ and $n\in \mathbb{N}$. By Fact \ref{fpegado}, the map $g:M\rightarrow N$ such that $g|_{C_n}=g_n$ for $n\in \mathbb{N}$, is a locally semialgebraic map. Let us show that $f$ is ld-homotopic to $g$ relative to $h$. The idea is to glue all the homotopies $H_n$ in a correct way. Let $t_n:=1-2^{-n}$ for each $n\in \mathbb{N}$. Consider the map $G:M\times I\rightarrow N$ such that (a) $G(x,t)=H_n(x,\frac{t-t_{n}}{t_{n+1}-t_n})$ for all $x\in M$ and $t\in [t_n,t_{n+1}]$ and (b) $G(x,t)=g(x)$ otherwise. By construction it is clear that $G(x,t)=h(x)$ for all $(x,t)\in C\times I$. It remains to check that $G$ is indeed an ld-map. By Fact \ref{fpegado}, it suffices to show that the restriction $G|_{Y_n\times I}$ is definable for each $n\in \mathbb{N}$. So fix $n\in \mathbb{N}$. By definition, $G|_{Y_n\times [0,t_n]}$ is clearly definable. On the other hand, take $(x,t)\in Y_n\times [t_n,1]$. If $t>t_m$ for every $m\in \mathbb{N}$, then $G(x,t)=g(x)$ by definition. If $t\in [t_{m},t_{m+1}]$ for some $m\geq n$, then $G(x,t)=H_m(x,t)=g_n(x)=g(x)$. Therefore $G|_{Y_n\times [t_n,1]}=g|_{Y_n}$, which is also a definable map. Hence $G|_{Y_n\times I}$ is definable, as required.
\end{proof}
The following corollary is the analogue (and it can be proved adapting its proof) of \cite[Cor.3.4]{07pBO} for \ld s. Recall the definition of the realization of an \ld \ in an elementary extension given before Theorem \ref{teo:invhomolld}.
\begin{cor}\label{cor:ominreallocall}Let $M$ and $N$ be two pairs of regular paracompact locally semialgebraic spaces defined without parameters. Then, there exist a bijection
$$\rho:[M(\mathbb{R}),N(\mathbb{R})] \rightarrow [M,N]^{\mathcal{R}},$$
where $[M(\mathbb{R}),N(\mathbb{R})]$ denotes the classical homotopy set. Moreover, if the real closed field $R$ is a field extension of $\mathbb{R}$, then the result remains true allowing parameters from $\mathbb{R}$.
\end{cor}
Note that both Theorem \ref{teo:locprinci} and Corollary \ref{cor:ominreallocall} remain true for systems of \ld s (see \cite[Cor.3.3]{07pBO}). Thanks to the Triangulation Theorem for \ld s (see Fact \ref{fact:triangulation}), we have also the following corollary (see the proof of \cite[Cor.3.6]{07pBO}, noting that the finiteness of the simplicial complexes plays an irrelevant role).
\begin{cor}Let $M$ and $N$ be \ld s defined without parameters. Then, any ld-map $f:M\rightarrow N$ is ld-homotopic to an ld-map $g:M\rightarrow N$ defined without parameters. If moreover $M$ and $N$ are locally semialgebraic spaces then $g$ can also be taken locally semialgebraic.
\end{cor}
\subsection{Homotopy groups of locally definable spaces}\label{shomotgrld}
The homotopy groups in the locally definable category are defined as in the definable setting using ld-maps instead of the definable ones (see Section 4 in \cite{07pBO}). Specifically, given a \textit{pointed \ld} \ $(M,x_0)$, i.e., $M$ is an \ld \ and $x_0\in M$, we define the $n$-\textbf{homotopy group} as the homotopy set $\pi_n(M,x_0)^{\mathcal{R}}:=[(I^n,\partial I^n),(M,x_0)]^{\mathcal{R}}$. We define $\pi_0(M,x_0)$ as the collection of all connected components of $M$ (which coincide with the collection of the path connected ones by Fact \ref{fact:conexoporarcos}). We say that $(M,A,x_{0})$ is a \textit{pointed pair of \ld s} if $M$ is an \ld, $A$ is an admissible subspace of $M$ and $x_0\in A$. The \textbf{relative $n$-homotopy group}, $n\geq 1$, of a pointed pair $(M,A,x_{0})$ of \ld s is the homotopy set $\pi_{n}(X,A,x_{0})^{\mathcal{R}}=[(I^{n},I^{n-1},J^{n-1}),(X,A,x_{0})]^{\mathcal{R}}$,
where $I^{n-1}=\{(t_{1},\ldots,t_{n})\in I^{n}:t_{n}=0 \}$ and
$J^{n-1}=\overline{\partial I^{n}\setminus I^{n-1}}$.

As in the definable case (see Section 4 in \cite{07pBO}), we can see that the homotopy groups $\pi_{n}(M,x_{0})^{\mathcal{R}}$ and 
$\pi_{m}(M,A,x_{0})^{\mathcal{R}}$ are indeed groups for $n\geq 1$ and $m\geq 2$, the group operation is defined via the usual concatenation of maps. Moreover, these groups are abelian for $n\geq 2$ and $m\geq 3$. Also, given an ld-map between pointed \ld s (or pointed pairs of \ld s), we define the induced map in homotopy, as usual, by composing. The latter will be a group homomorphism in the case we have a group structure. It is easy to check that with these definitions of homotopy group and induced map, both the absolute and relative homotopy groups $\pi_n(-)$ are covariant functors.

The following three results (and their relative versions) can be deduced from Theorem \ref{teo:locprinci} (see the proofs of \cite[Thm.4.1]{07pBO}, \cite[Cor.4.3]{07pBO} and \cite[Cor.4.4]{07pBO}).

\begin{cor}\label{rhoesismorfis}For every regular paracompact locally semialgebraic pointed space $(M,x_0)$ and every $n\geq 1$, the map $\rho:\pi_{n}(M,x_0)^{\mathcal{R}_{0}}\rightarrow \pi_{n}(M,x_0)^{\mathcal{R}}:[f]\mapsto [f]$, is a natural isomorphism. 
\end{cor}
\begin{cor}\label{cor:homtopiadefigualtopo}Let $(M,x_{0})$ be a regular paracompact locally semialgebraic pointed space defined without parameters. Then, there exists a natural isomorphism between the classical homotopy group
$\pi_{n}(M(\mathbb{R}),x_{0})$ and the homotopy group 
$\pi_{n}(M(R),x_{0})^{\mathcal{R}}$ for every $n\geq 1$.
\end{cor}
\begin{cor}\label{cinvextexp}The homotopy groups are invariants under elementary extensions and o-minimal expansions.
\end{cor}

All the results of Section 4 in \cite{07pBO} remains true in the locally definable category. We recall here briefly these results.\\

\vspace*{-0.3cm}
\hspace{-0.6cm}(1) \textit{The homotopy property}: If two ld-maps are ld-homotopic then they induce the same homomorphism between the homotopy groups.\\

\vspace*{-0.3cm}
\hspace{-0.6cm}(2) \textit{The exactness property}: Given a pointed pair $(M,A,x_0)$ of \ld s, the following sequence is exact,
\begin{small}$$ \cdots \rightarrow \pi_n(A,x_0)\stackrel{i_*}{\rightarrow} \pi_n(M,x_0)\stackrel{j_*}{\rightarrow} \pi_n(M,A,x_0)\stackrel{\partial}{\rightarrow} \pi_{n-1}(A,x_0) \rightarrow \cdots \rightarrow\pi_0(A,x_0),$$
\end{small}
\hspace{-0.2cm}where $\partial$ is the usual boundary map $\partial: \pi_n(M,A,x_0)^{\mathcal{R}}\rightarrow \pi_{n-1}(A,x_0)^{\mathcal{R}}:[f]\mapsto [f|_{I^{n-1}}]$ and $i:(A,x_0)\rightarrow (M,x_0)$ and $j:(M,x_0,x_0)\rightarrow (M,A,x_0)$ are the inclusions (and the superscript $\mathcal{R}$ has been omitted).\\

\vspace*{-0.3cm}
\hspace{-0.6cm}(3) \textit{The action of $\pi_1$ on $\pi_n$}: Given a pointed \ld \ $(M,x_0)$, there is an action $\beta:\pi_1(M,x_0)^{\mathcal{R}}\times \pi_n(M,x_0)^{\mathcal{R}}\rightarrow \pi_n(M,x_0)^{\mathcal{R}}$ such that $\beta_{[u]}:=\beta([u],-):\pi_{n}(M,x_{0})^{\mathcal{R}} \rightarrow  \pi_{n}(M,x_{0})^{\mathcal{R}}$ is an isomorphism for every $[u]\in \pi_1(M,x_0)^{\mathcal{R}}$. In a similar way, given a pointed pair $(M,A,x_0)$ of \ld s, there is an action $\beta:\pi_1(A,x_0)^{\mathcal{R}}\times \pi_n(M,A,x_0)^{\mathcal{R}}\rightarrow \pi_n(M,A,x_0)^{\mathcal{R}}$ such that $\beta_{[u]}:=\beta([u],-):\pi_{n}(M,A,x_{0})^{\mathcal{R}} \rightarrow  \pi_{n}(M,A,x_{0})^{\mathcal{R}}$ is an isomorphism for every $[u]\in \pi_1(A,x_0)^{\mathcal{R}}$.\\

\vspace*{-0.3cm}
The homotopy property is clear by definition. The exactness property can be proved with a straightforward adaptation of the proof of the classical one. Alternatively, we can also transfer the classical exactness property using the Triangulation Theorem (see Fact \ref{fact:triangulation}) and Corollary \ref{rhoesismorfis}. Finally, the existence of the action of $\pi_1$ on $\pi_n$ is just an application of the homotopy extension lemma (see Lemma \ref{lema:extlochom} and \cite[Prop.4.6.3)]{07pBO}). Furthermore, the following technical lemma is easy to prove (see the proof of \cite[Lem.4.7]{07pBO}).
\begin{lema}\label{lema:empujacurvas}Let $(M,x_0)$ and $(N,y_0)$ two pointed \ld s. Let $\psi:(M,x_{0}) \rightarrow (N,y_{0})$ be an ld-map and let $[u]\in \pi_{1}(M,x_{0})^{\mathcal{R}}$. Then, for all $[f]\in \pi_{n}(M,x_0)^{\mathcal{R}}$, $\psi_{*}(\beta_{[u]}([f]))=\beta_{\psi_{*}([u])}(\psi_{*}([f]))$.
\end{lema}
The only part of Section 4 in \cite{07pBO} which has not an obvious extension to \ld s is the one which concerning fibrations. Naturally, we say that an ld-map $p:E\rightarrow B$ between \ld s is a \textbf{(Serre) fibration} if it has the homotopy lifting property for each (resp. closed and bounded) definable sets. As in \cite[Rmk. 4.8]{07pBO}, the homotopy lifting property for closed simplices implies the homotopy lifting property for pairs of closed and bounded definable sets. Note that the restriction of a (Serre) fibration to the preimage of a definable subspace is not necessarily a definable (resp. Serre) fibration. However, the fibration property (see \cite[Thm.4.9]{07pBO}) for \ld s can be proved just adapting directly the classical proof.
\begin{teo}[The fibration property]\label{teo:fibrationproperty} Let $B$ and $E$ be \ld s. Then, for every Serre fibration $p:E\rightarrow B$, the induced map $p_*:\pi_n(E,F,e_0)^{\mathcal{R}}\rightarrow \pi_n(B,b_0)^{\mathcal{R}}$ is a bijection for $n=1$ and an isomorphism for all $n\geq 2$, where $e_0\in F=p^{-1}(b_0)$. 
\end{teo}

As in the definable setting (see \cite[Prop.4.10]{07pBO}), the main examples of fibrations are the covering maps. Given two \LD s $E$ and $B$, a \textbf{covering map} $p:E\rightarrow B$ is a surjective ld-map $p$ such that there is an admissible covering $\{U_i:i\in I\}$ of $B$ by open definable subspaces and for each $i\in I$ and each connected component $V$ of $p^{-1}(U_i)$, the restriction $p|_{V}:V\rightarrow U_i$ is a locally definable homeomorphism (so in particular both $V$ and $p|_V$ are definable).
\begin{prop}\label{prop:recufibracionlocal}Let $B$ and $E$ be \ld s. Then, every covering map $p:E\rightarrow B$ is a fibration.
\end{prop}
\begin{proof}Firstly, note that coverings satisfy the unicity of liftings as in the definable case (see \cite[Lem.2.5]{04EO}). Indeed, given a connected \ld \ $Z$ and two ld-maps $\widetilde{f}_1,\widetilde{f}_2:Z\rightarrow E$ with $p\circ \widetilde{f}_1=p\circ \widetilde{f}_2$ and $\widetilde{f}_1(z)=\widetilde{f}_2(z)$ for some $z\in Z$, we have that $\widetilde{f}_1=\widetilde{f}_2$. This is so because both $\{z\in Z: \widetilde{f}_1(z)=\widetilde{f}_2(z)\}$ and $\{z\in Z: \widetilde{f}_1(z)\neq \widetilde{f}_2(z)\}$ are clopen admissible subspaces of $Z$. The path lifting and the homotopy lifting properties also remain true for $p$ (see the definable case in \cite[Prop.2.6]{04EO} and \cite[Prop.2.7]{04EO}). To see this for the path lifting property take an admissible covering $\{U_j:j \in J \}$ of $B$ as in the definition of covering map. Let $\gamma:I \rightarrow B$ be an ld-curve. Since $\gamma(I)$ is a definable subspace of $B$, we have that $\gamma(I)\subset \bigcup_{j\in J_0}U_j$ for some finite subset $J_0$ of $J$. Now, by the shrinking covering property of definable sets, there are $0=s_0<s_1<\cdots < s_r=1$ such that for each $i=0,\ldots,r-1$ we have $\gamma([s_i,s_{i+1}])\subset U_{j(i)}$ and $\gamma(s_{i+1})\in U_{j(i)}\cap U_{j(i+1)}$. Hence, by the unicity of liftings, it suffices to lift each $\gamma|_{[s_i,s_{i+1}]}$ step by step using the definable homeomorphism $p|_{V_{j(i)}}:V_{j(i)}\rightarrow U_{j(i)}$ for the suitable connected component $V_{j(i)}$ of $p^{-1}(U_{j(i)})$. The proof of the homotopy lifting property is similar.

Finally, the above properties and the fact that the images of definable sets by ld-maps are definable subspaces, give us the homotopy lifting property for definable sets as in \cite[Prop.4.10]{07pBO}.
\end{proof}
\begin{cor}\label{cgrcovld}Let $B$ and $E$ be \ld s. Let $p:E\rightarrow B$ be a covering and let $p(e_0)=b_0$. Then, $p_*:\pi_n(E,e_0)^{\mathcal{R}}\rightarrow \pi_n(B,b_0)^{\mathcal{R}}$ is an isomorphism for every $n>1$ and injective for $n=1$.
\end{cor}
\begin{proof}Since $p$ is a covering, $p^{-1}(b_0)$ is discrete. Hence $\pi_n(p^{-1}(b_0),e_0)=0$ for every $n\geq 1$. Then, the result follows from Proposition \ref{prop:recufibracionlocal} and both the exactness and the fibration properties.
\end{proof}

We end this subsection with one the motivations for considering the locally definable category. 
\begin{fact}\emph{\cite[ Thm.5.11]{84DK}}\label{fact:recubrimientos} Let $B$ be a connected \LD, $b_0\in B$ and let $L$ be a subgroup of $\pi_1(B,b_0)^{\mathcal{R}}$. Then, there exists connected \LD \ $E$ and a covering  $p:E\rightarrow B$ with $p_*(\pi_1(E,e_0)^{\mathcal{R}})=L$ for some $e_0\in p^{-1}(b_0)$. Moreover, if $B$ is an \ld \ then $E$ is also an \ld.
\end{fact}

We give in Appendix \ref{sappcov} a proof of this fact for \ld s, for completeness.
\subsection{The Hurewicz and Whitehead theorems for locally definable spaces}\label{shurewiczld}
We define the Hurewicz homomorphism in a similar same way as in the definable case but using the homology groups developed in Section \ref{shomol}.  We fix a generator $z_n^{\mathcal{R}_0}$ of $H_n(I^n,\partial I^n)^{\mathcal{R}_0}$ (recall that $H_n(I^n,\partial I^n)^{\mathcal{R}_0}\cong\mathbb{Z}$). Let $z_n^{\mathcal{R}}:=\theta(z_n^{\mathcal{R}_0})$, where $\theta$ is the natural transformation of Notation \ref{nota:homolsaomin} between both the (locally) semialgebraic and the (locally) definable homology groups. Given a pointed \ld \ $(M,x_0)$, the \textit{\textbf{Hurewicz homomorphism}}, for $n\geq 1$, is the map $h_{n,\mathcal{R}}:\pi_{n}(M,x_{0})^{\mathcal{R}} \rightarrow H_{n}(M)^{\mathcal{R}}:[f] \mapsto h_{n,\mathcal{R}}( [f] )=f_{*}(z_{n}^{\mathcal{R}})$, where $f_{*}:H_{n}(I^{n},\partial I)^{\mathcal{R}}\rightarrow
H_{n}(M)^{\mathcal{R}}$ denotes the map in singular homology induced by $f$. We define the relative Hurewicz homomorphism adapting in the obvious way what was done in the absolute case. It is easy to check that $h_{n,\mathcal{R}}$ is a natural transformation between the functors $\pi_{n}(-)^{\mathcal{R}}$ and
$H_{n}(-)^{\mathcal{R}}$. The following result can be easily deduced from the naturality of the isomorphisms $\rho$ and $\theta$ introduced in Corollary \ref{rhoesismorfis} and Notation \ref{nota:homolsaomin} respectively (see \cite[Prop.5.1]{07pBO}).
\begin{prop}\label{lema:homohurdefsa}Let $(M,x_{0})$ be a pointed regular paracompact locally semialgebraic space. Then, the following diagram commutes
\begin{displaymath}
\xymatrix{ \pi_{n}(M,x_{0})^{\mathcal{R}_{0}}
\ar[r]^{h_{n,\mathcal{R}_{0}}} \ar[d]_{\rho} &
     H_{n}(M)^{\mathcal{R}_{0}} \ar[d]^{\theta}\\
\pi_{n}(M,x_{0})^{\mathcal{R}} \ar[r]_{h_{n,\mathcal{R}}} &
H_{n}(M)^{\mathcal{R}}}
\end{displaymath}
for all $n\geq 1$.
\end{prop}

Now, the proofs in the definable setting of the Hurewicz and the Whitehead theorems (see \cite[Thm.5.3]{07pBO} and \cite[Thm.5.6]{07pBO}) apply for \ld s just using (i) the locally definable category instead of the definable one, (ii) the respective isomorphisms $\rho$ and $\theta$ of Theorem \ref{teo:locprinci} and Notation \ref{nota:homolsaomin} instead of the definable ones and (iii) the Triangulation Theorem for \ld s (see Fact \ref{fact:triangulation}). Note that in the proofs of the definable versions of the Hurewicz and Whitehead theorems, the finiteness of the simplicial complexes plays an irrelevant role. Specifically, we have the following results (recall the action $\beta$ of $\pi_1$ on $\pi_n$ defined after Corollary \ref{cinvextexp}).
\begin{teo}[\textbf{Hurewicz theorems}]\label{teo:Hureabsolocal}Let $(M,x_{0})$ be a pointed \ld \ and $n \geq 1$. Suppose that
$\pi_{r}(M,x_{0})^{\mathcal{R}}=0$ for  every $0\leq r \leq n-1$.
Then, the Hurewicz homomorphism $$h_{n,\mathcal{R}}:\pi_{n}(M,x_{0})^{\mathcal{R}}\rightarrow
H_{n}(M)^{\mathcal{R}}$$ is surjective and its kernel is the subgroup generated by $\{\beta_{[u]}([f])[f]^{-1}: [u]\in
\pi_{1}(M,x_{0})^{\mathcal{R}},[f]\in \pi_{n}(M,x_{0})^{\mathcal{R}}\}$. In particular, $h_{n,\mathcal{R}}$ is an isomorphism for $n\geq 2$.
\end{teo}
\begin{teo}[\textbf{Whitehead theorem}]\label{whiteheadlocal}Let $M$ and $N$ be two \ld s. Let $\psi:M\rightarrow N$ be an ld-map such that for some $x_0\in M$, $\psi_{*}:\pi_{n}(M,x_{0})^{\mathcal{R}}\rightarrow \pi_{n}(N,\psi(x_{0}))^{\mathcal{R}}$ is an isomorphism for all $n\geq 1$. Then, $\psi$ is an ld-homotopy equivalence.
\end{teo}
\begin{cor}\label{cor:contraclocal}Let $M$ be an \ld \ and let $x_0\in M$. If $\pi_n(M,x_0)^{\mathcal{R}}=0$ for all $n\geq 0$ then $M$ is ld-contractible.
\end{cor}
\section{Appendix}
\subsection{The Triangulation Theorem}\label{sappepr}
Below we sketch the proof of the Triangulation theorem for \ld s (see Fact \ref{fact:triangulation}). The main step is Fact \ref{fembed}, whose proof will be given at the end. The advantage of working with a partially complete \ld s is that given a triangulation of a closed definable subspace of itself we know that the corresponding simplicial complex must be closed. This allow us to prove the following ``glue'' principle of triangulations for partially complete \ld s. Firstly, recall that given two ld-triangulations $(K,\phi)$ and $(L,\psi)$ of a \ld \ $M$, we say that $(K,\phi)$ \textit{\textbf{refines}} $(L,\psi)$ if for every $\tau\in L$ there is $\sigma\in K$ such that $\phi(\sigma)\subset \psi(\tau)$. We say that $(K,\phi)$ is \textit{\textbf{equivalent}} to $(L,\psi)$ if each ld-triangulation is a refinement of the other.
\begin{fact}\emph{\cite[Thm. II.4.1]{85DK}}\label{fpegtrian} Let $M$ be a partially complete \ld \ and $\{C_j:j\in J\}$ a locally finite covering of $M$ by closed definable subspaces. Let $(K_j,\phi_j)$ be a triangulation of $C_j$ for each $j\in J$. Moreover, assume that for every $i,j\in J$ with $C_i\cap C_j\neq \emptyset$, $\phi_i$ and $\phi_j$ are equivalent on $C_i\cap C_j$. Then, there is a ld-triangulation $(K,\phi)$ of $M$ such that $\phi$ is equivalent to $\phi_j$ on $C_j$ for each $j\in J$.
\end{fact}
\begin{proof} Since $M$ is partially complete and each $C_j$ is closed, we have that $K_j$ is a closed simplicial complex for each $j\in J$. Denote by $E$ the quotient of the disjoint union of the sets $Vert(K_j)$ of vertices of each $K_j$ by the equivalence relation such that $v\in Vert(K_i)$ and $w\in Vert(K_j)$ are related if and only if $\phi_i(v)=\phi_j(w)$. Clearly, for each $j\in J$ we have an injective map $I_j:Vert(K_j)\rightarrow E$. Let $S:=\{ \{I_j(v_0),\ldots,I_j(v_n)\}:(v_0,\ldots,v_n)\in K_j,j\in J\}$. It is easy to check that $(E,S)$ is an abstract simplicial complex. In fact, since the covering $\{C_j:j\in J\}$ is locally finite, the complex $(E,S)$ is locally finite. Consider a realization $|K|$ of $(E,S)$ and for each $j\in J$ denote by $|I_j|:|K_j|\rightarrow |K|$ the simplicial map generated by $I_j$. Finally, consider the map $\phi:|K|\rightarrow M$ such that $\phi|_{Y_j}=\phi_j\circ |I_j|^{-1}$. Since the triangulations $(K_j,\phi_j)$ are equivalent on the intersections, $\phi$ is a well-defined ld-map. It is not difficult to prove that $(K,\phi)$ is the required ld-triangulation of $M$.
\end{proof}
\begin{proof}[Proof of Fact \ref{fact:triangulation}]  By Fact \ref{fembed} we can assume that $M$ is partially complete. We can also assume that $M$ is connected. Therefore, since $M$ is paracompact, $M$ is Lindel\"of (see Fact \ref{fact:claudef}). Hence, there is a covering $\{C_n:n\in \mathbb{N}\}$ of $M$ by closed definable subspaces such that $C_n\cap C_m=\emptyset$ if $|n-m|>1$. Indeed, it suffices to use the shrinking property of coverings (see \cite[Thm. I.4.11]{85DK}) with the locally finite covering constructed in the proof of Fact \ref{fact:claudef}.(3). Note that for each $n\in \mathbb{N}$ there is only a finite number of $j\in J$ such that $A_j\cap C_n\neq \emptyset$. Therefore, by the (affine) Triangulation theorem and applying an iteration process, there are triangulations $(K_n,\phi_n)$ of $C_n$ partitioning $C_n\cap M_{n+1}$, $\{C_n\cap A_j \}_{j\in J}$ and $\{\phi_{n-1}(\sigma)\cap C_n\}_{\sigma\in K_{n-1}}$. Note that  $(K_n,\phi_n)$ refines $(K_{n-1},\phi_{n-1})$ on $C_n\cap C_{n-1}$. Now, since $C_n\cap C_{n-1}$ and $C_n\cap C_{n+1}$ are disjoint, there is a triangulation $(L_n,\psi_n)$ of $C_n$ refining $(K_n,\phi_n)$ and equivalent to $(K_n,\phi_n)$ and $(K_{n+1},\phi_{n+1})$ on $C_n\cap C_{n-1}$ and $C_n\cap C_{n+1}$ respectively (see \cite[Lem.I.4.3]{85DK}). Finally, by Fact \ref{fpegtrian}, there is an ld-triangulation $(L,\psi)$ of $M$ partitioning $\{A_j:j\in J\}$. 
\end{proof}
Now, we prove Fact \ref{fembed}. The following result states a glueing principle of definable spaces with closed intersections.
\begin{fact} \emph{\cite[Thm. II.1.3]{85DK}}\label{pegadoespacios} Let $M$ be a set and $\{M_i:i\in I\}$ a family of subsets of $M$. Assume that for each $i\in I$, $M_i$ has an affine definable space structure satisfying that \\
\emph{(i)} $M_i\cap M_j$ is a closed definable subspace of both $M_i$ and $M_j$ for every $i,j\in I$ and the structure that $M_i\cap M_j$ inherits from $M_i$ and $M_j$ are equal and,\\
\emph{(ii)} the family $\{M_i\cap M_j:j\in I\}$ is finite for every $i\in I$.\\
Then, there is a (unique) \ld \ structure in $M$ such that\\
\emph{(a)} $M_i$ is a closed definable subspace of $M$ for every $i\in I$;\\
\emph{(b)} the structure that $M_i$ inherits from $M$ equals its affine structure and,\\
\emph{(c)} the family $\{M_i:i\in I\}$ is locally finite. 
\end{fact}
\begin{proof}We just give the ideas of the case $I=\{1,2\}$, the general proof can be found in \cite{85DK}. Denote by $\psi_i:M_i\rightarrow E_i\subset R^n$ the chart of $M_i$ for $i=1,2$. Let $A=M_1\cap M_2$. By Tietze extension theorem (see  \cite[Ch.8,Cor.3.10]{98Dr}) there is a definable map $\chi_1:M_1\rightarrow R^n$ such that $\chi_1|_A=\psi_2$. Similarly, there is a definable map $\chi_2:M_2\rightarrow R^n$ such that $\chi_2|_A=\psi_1$. Consider the map $\phi_1:M\rightarrow R^n$ such that $\phi_1|_{M_1}=\psi_1$ and $\phi_1|_{M_2}=\chi_2$. Consider also the map $\phi_2:M\rightarrow R^n$ such that $\phi_2|_{M_2}=\psi_2$ and $\phi_2|_{M_1}=\chi_1$. On the other hand, by \cite[Ch.6, Lemma 3.8]{98Dr}, there are definable functions $h_1:M_1\rightarrow [-1,0]$ and $h_2:M_2\rightarrow [0,1]$ such that $h_1^{-1}(0)=h_2^{-1}(0)=A$. Consider the function $h:M\rightarrow [-1,1]$ such that $h|_{M_1}=h_1$ and $h|_{M_2}=h_2$. Note that $h^{-1}(0)=A$, $h^{-1}([-1,0])=M_1$ and $h^{-1}([0,1])=M_2$. Finally, consider the map $f:M\rightarrow R^n\times R^n \times R:x\mapsto (\phi_1(x),\phi_2(x),h(x))$. Note that the function $f$ is injective and the map $\psi_i^{-1}\circ f:E_i\rightarrow R^{2n+1}$ is definable for $i=1,2$. Hence $N_1:=f(M_1)$ and $N_2:=f(M_2)$ are definable. In particular, $E:=f(M)=N_1\cup N_2$ is definable. Using the bijection $f:M\rightarrow E$, we define a structure of affine definable space in $M$. Next, we check that properties (a) and (b) hold. Firstly, note that $N_1=\{(x_1,\ldots,x_{2n+1})\in E:x_{2n+1}\leq 0\}$ and $N_2=\{(x_1,\ldots,x_{2n+1})\in E:x_{2n+1}\geq 0\}$ and therefore  $N_1$ and $N_2$ are closed subsets of $E$. Finally, $f|_{M_i}:M_i\rightarrow N_i$ is an embedding for $i=1,2$. Indeed, it suffices to observe that $f^{-1}|_{N_i}=\psi_1^{-1}\circ pr$ is definable, where $pr$ denotes the projection over the first $n$ coordinates. In a similar way, we prove that $f^{-1}|_{N_2}$ is also definable.
\end{proof}
\begin{proof}[Proof of Fact \ref{fembed}] Let $\{M_i:i\in I\}$ be a locally finite covering of $M$ by open definable subspaces. For each $i\in I$ there is an sphere $S_i:=S^{n_i}$, $n_i\in \mathbb{N}$, and an ld-map $g_i:M\rightarrow S_i$ with $g_i^{-1}(p_i)=M\setminus M_i$, $p_i$ the north pole of $S^{n_i}$, and such that $g_i|_{M_i}$ is an embedding (see \cite[Lem.,II.2.2]{85DK}). On the other hand, we define the finite subsets of indexes $\Gamma_1(i):=\{j\in I:M_i\cap M_j\neq \emptyset \}$, $\Gamma_2(i):=\bigcup_{j\in \Gamma_1(i)}\Gamma_1(j)$ and $\Gamma^{*}_{2}(i)=\Gamma_2(i)\setminus \{i\}$ for each $i\in I$. Consider the set $Z:=\prod_{i\in I}(S_i\times [0,1])$ and the family of subsets $N_i:=\prod_{j\in I}N_{i,j}\subset Z$, where $N_{i,i}:=S_i\times \{1\}$, $N_{i,j}:=(p_j,0)$ if $j\in I\setminus \Gamma_2(i)$ and $N_{i,j}:=S_{j}\times [0,1]$ if $j\in \Gamma^*_2(i)$. We regard each $N_i$ in the obvious way as a definably compact definable space isomorphic to the product $(S_i\times \{1\}) \times \prod_{j\in \Gamma^*_2(i)}(S_j\times[0,1])$. Now, by Theorem \ref{pegadoespacios}, we have a partially complete \LD \ structure on  $$N:=\bigcup_{i\in I}N_i$$ such that each $N_i$ is closed in $N$ (with the inherited structure of definable space from $N$ equal to the original one) and such that $\{N_i:i\in I\}$ is a locally finite covering of $N$. Indeed, it suffices to check that given $i\in I$, there are only a finite number of $j\in I$ with $N_i\cap N_j\neq \emptyset$ and, in this case, $N_i\cap N_j$ is closed in both $N_i$ and $N_j$ with the inherited definable space structures equal. But clearly, $N_i\cap N_j\neq\emptyset$ if and only if $i\in \Gamma_2(j)$ and in this case  $N_i\cap N_j=\prod_{k\in I}N_{i,j,k}$, where $N_{i,j,k}:=S_k\times \{ 1\}$ if $k=i$ or $k=j$, $N_{i,j,k}:=S_{k}\times [0,1]$ if $k\in \Gamma^*_2(i)\cap \Gamma^*_2(j)$ and $N_{i,j,k}:=(p_k,0)$ in other case. In fact, $N$ is partially complete because given a closed definable subspace $X$ of $N$ we have that $X=(X\cap N_{i_1})\cup \cdots \cup (X\cap N_{i_m})$ for some $i_1,\ldots,i_m\in I$ and each $X\cap N_{i_1}$,\ldots,$X\cap N_{i_m}$ is a definably compact definable space (since it is a closed subset of a definably compact one).

Now, we construct the embedding of $M$ in $N$. Denote by $W_i=\bigcup_{j\in \Gamma_1(i)}M_j$, which is an open definable subspace of $M$. Note that $\overline{M_i}\subset W_i$. By \cite[Th. I.4.15]{85DK}, for each $i\in I$ there is an ld-function $f_i:M\rightarrow [0,1]$ such that $f_i^{-1}(1)=\overline{M_i}$ and $f_i^{-1}(0)=M\setminus W_i$. Finally, consider the map $$\psi:M\rightarrow N:x\mapsto (g_i(x),f_i(x))_{i\in I}.$$
Note that $\psi$ is a well-defined injective ld-map. A straightforward adaptation of \cite[Thm. II.2.1]{85DK} shows that $\psi(M)$ is an admissible subspace of $N$ and that $\psi:M\rightarrow \psi(M)$ is an ld-homeomorphism.\end{proof}
\subsection{Covering maps for \ld s}\label{sappcov}
\begin{proof}[Proof of Fact \ref{fact:recubrimientos}] Consider the collection $\mathcal{P}$ of all ld-curves $\alpha:I\rightarrow B$ such that $\alpha(0)=b_0$. Let $\thicksim$ be the equivalence relation on $\mathcal{P}$ such that $\alpha \thicksim \beta$ if and only if $\alpha(1)=\beta(1)$ and $[\alpha * \beta^{-1}]\in L$, where $*$ denotes the usual concatenation of curves. We will denote by $\alpha^{\#}$ the class of $\alpha\in \mathcal{P}$. Let $E=\mathcal{P}/ \sim$ and $p:E\rightarrow B:\alpha^{\#} \mapsto \alpha(1)$. Now, we divide the proof in several steps.\\
\hspace*{0.5cm}\texttt{(1)} $E$ is an \LD: Firstly, note that every definable subspace of $B$ has a finite covering by open connected definable subspaces which are simply connected (because of Remark \ref{obs:regularity}, the Triangulation theorem and the fact that the star of a vertex is definably simply connected). Therefore, since $B$ is an \ld, there exist a locally finite covering $\{U_j:j\in J\}$ of $B$ such that each $U_j$ is a connected and simply connected (i.e, $\pi_1(U_j)^{\mathcal{R}}=0$) definable open subspace of $B$. Now, for each $j\in J$ and $\alpha\in \mathcal{P}$ with $\alpha(1)\in U_j$, we define $W_{j,\alpha}:=\{(\alpha * \delta)^{\#}:\delta:I\rightarrow U_j \textrm{ ld-map}, \delta(0)=\alpha(1)\}$. Henceforth, when we write $W_{j,\alpha}$, we assume that $\alpha(1)\in U_j$. Consider the map $\phi_{j,\alpha}:W_{j,\alpha}\rightarrow U_j: (\alpha * \delta)^{\#} \mapsto \delta(1)$ for each $j\in J$ and $\alpha\in \mathcal{P}$. Since $U_j$ is connected and simply connected, $\phi_{j,\alpha}$ is a well-defined bijection for every $j\in J$ and $\alpha\in \mathcal{P}$. The family $(W_{j,\alpha},\phi_{j,\alpha})_{j\in J,\alpha\in \mathcal{P}}$ is an atlas of $E$. Indeed, fix $i,j\in J$ and $\alpha,\beta\in \mathcal{P}$ with $W_{i,\alpha}\cap W_{j,\beta}\neq \emptyset$. Then, $\phi_{i,\alpha}(W_{i,\alpha}\cap W_{j,\beta})$ is the union of some connected components of $U_i\cap U_j$. Moreover, $\phi_{j,\beta}(W_{i,\alpha}\cap W_{j,\beta})$ is the union of exactly the same connected components of $U_i\cap U_j$, i.e., $\phi_{j,\beta}(W_{i,\alpha}\cap W_{j,\beta})=\phi_{i,\alpha}(W_{i,\alpha}\cap W_{j,\beta})$. This shows that both $\phi_{i,\alpha}(W_{i,\alpha}\cap W_{j,\beta})$ and $\phi_{j,\beta}(W_{i,\alpha}\cap W_{j,\beta})$ are open in $U_i$ and $U_j$ respectively and that each change of charts is the identity, hence definable.\\
\hspace*{0.5cm}\texttt{(2)} The map $p$ is an ld-map: since $p|_{W_{j,\alpha}}:W_{j,\alpha}\rightarrow U_j\subset B$ is a definable map of definable spaces, for all $W_{j,\alpha}$.\\
\hspace*{0.5cm}\texttt{(3)} $E$ is paracompact: Fix  $i\in J$ and $\alpha\in \mathcal{P}$. We prove that $\#\{W_{j,\beta}: W_{i,\alpha}\cap W_{j,\beta}\neq\emptyset, j\in J,\beta\in \mathcal{P}\}$ is finite. Firstly, note that if $W_{i,\alpha}\cap W_{j,\beta}\neq \emptyset$ then $U_i\cap U_j \neq \emptyset$. Therefore, since the covering $\{U_j:j\in J\}$ is locally finite, it suffices to prove that the family $\{W_{j,\beta}:W_{i,\alpha}\cap W_{j,\beta}\neq\emptyset, \beta\in \mathcal{P}\}$ is finite for a fixed $j\in J$. Indeed, we will show that given $W_{j,\beta_1}$ and $W_{j,\beta_2}$ with $W_{i,\alpha}\cap W_{j,\beta_1}\neq\emptyset$ and $W_{i,\alpha}\cap W_{j,\beta_2}\neq\emptyset$, if $p(W_{i,\alpha}\cap W_{j,\beta_1})\cap p(W_{i,\alpha}\cap W_{j,\beta_2}) \neq\emptyset$ then $W_{j,\beta_1}=W_{j,\beta_2}$. The latter is enough because for each $\beta\in \mathcal{P}$, $p(W_{i,\alpha}\cap W_{j,\beta})$ ($=\phi_{i,\alpha}(W_{i,\alpha}\cap W_{j,\beta})$) is the union of some connected components of $U_i\cap U_j$, which has only a finite number of them. Firstly,  since $U_j$ is connected, it is easy to prove that if $\gamma^{\#}\in W_{j,\beta_1}$ then $W_{j,\gamma}=W_{j,\beta_1}$. The same holds for $W_{j,\beta_2}$. So, if $W_{j,\beta_1}\cap W_{j,\beta_2}\neq \emptyset$ then $W_{j,\beta_1}= W_{j,\beta_2}$. On the other hand, since $p|_{W_{i,\alpha}}=\phi_{i,\alpha}$ and $\phi_{i,\alpha}$ is a bijection, from $p(W_{i,\alpha}\cap W_{j,\beta_1})\cap p(W_{i,\alpha}\cap W_{j,\beta_2}) \neq\emptyset$ we deduce that $\emptyset \neq W_{i,\alpha}\cap W_{j,\beta_1}\cap W_{j,\beta_2}\subset W_{j,\beta_1}\cap W_{j,\beta_2}$ and hence $W_{j,\beta_1}= W_{j,\beta_2}$.\\
\hspace*{0.5cm}\texttt{(4)} The ld-map $p:E\rightarrow B$ is a covering map: By to proof of \texttt{(3)}, we have that $p^{-1}(U_j)=$ $\bigcup\hspace{-2.8mm\cdot}\hspace{1mm}$ $_{_{\alpha\in \mathcal{P}}}W_{j,\alpha}$ for every $j\in J$. On the other hand, $p|_{W_{j,\alpha}}:W_{j,\alpha}\rightarrow U_j$ is an ld-homeomorphism for every $j\in J$ and $\alpha\in \mathcal{P}$.\\
\hspace*{0.5cm}\texttt{(5)} $E$ is an \ld : Indeed, the regularity of $E$ can be deduced from the regularity of $B$ and \texttt{(4)}.\\
\hspace*{0.5cm}\texttt{(6)} $E$ is path-connected, hence connected: Let $e_0:=c_{b_0}^{\#}\in E$ for the ld-curve $c_{b_0}:I\rightarrow B: t\mapsto b_0$ (recall $b_0\in B$ is a fixed point). Given $\alpha\in \mathcal{P}$, we will show that there is and ld-map from $e_0$ to $\alpha^{\#}$. Consider the map $\widetilde{\alpha}:I\rightarrow E:s\mapsto \alpha_s^{\#}$, where $\alpha_s:I\rightarrow B:t\mapsto \alpha_s(t)=\alpha(ts)$ is clearly an ld-curve. Note that $p\circ \widetilde{\alpha}(s)=\alpha(s)$, $\widetilde{\alpha}(0)=e_0$ and $\widetilde{\alpha}(1)=\alpha^{\#}$. Let us check that $\widetilde{\alpha}$ is an ld-curve. Since $\alpha$ is an ld-curve, there are $s_0=0<s_1<\cdots <s_m=1$ such that $\alpha([s_k,s_{k+1}])\subset U_{i_k}$ for every $k=0,\ldots,m-1$. Hence $\widetilde{\alpha}(I)\subset \bigcup_{k=0}^{m-1} W_{i_{k},\alpha_{s_k}}$. On the other hand, $\phi_{i,\alpha_{s_k}} \circ \widetilde{\alpha}|_{[s_k,s_{k+1}]}=\alpha|_{[s_k,s_{k+1}]}$ for every $k=0,\ldots,m-1$ and therefore $\widetilde{\alpha}$ is an ld-curve as required.\\
\hspace*{0.5cm}\texttt{(7)} Finally, let us show that $p_*(\pi_1(E,e_0)^{\mathcal{R}})=L$. Let $\alpha$ be an ld-loop of $B$ at $b_0$. By the proof of \texttt{(6)}, $\widetilde{\alpha}:I\rightarrow E:s\mapsto \alpha_s^{\#}$, where $\alpha_s:I\rightarrow B:t\mapsto \alpha_s(t)=\alpha(ts)$, is an ld-curve. Now, as in the classical case, we have that $[\alpha]\in p_*(\pi_1(E,e_0)^{\mathcal{R}})$ if and only if $\widetilde{\alpha}(1)=\alpha^{\#}=e_0$. Indeed, the latter can be proved using both the path and homotopy lifting properties of covering maps (see the proof of Proposition \ref{prop:recufibracionlocal}). Hence $[\alpha]\in p_*(\pi_1(E,e_0)^{\mathcal{R}})$ if and only if $[\alpha]\in L$.\end{proof}
Note that if $B$ is an LD-group (see Definition \ref{dldgr}), then it is possible to define a group operation in the covering space $E$. Using the notation of the proof of Fact \ref{fact:recubrimientos}, given $\alpha,\beta \in \mathcal{P}$, we define $\alpha^{\#} \beta^{\#}:=(\alpha \beta)^{\#}$. Note that with this group operation $E$ becomes an LD-group. This was also proved in \cite{06EEl} for the particular case of the universal covering map of a definable group for o-minimal expansions of ordered groups.
\bigskip

POSTSCRIPT. After a preliminary version of this paper was written, the preprint \cite{08pP} by A. Pi\c{e}kosz  has appeared with some related results.

\begin{footnotesize}
\end{footnotesize}

\vspace{0.6cm}
\hspace{-0.6cm}\begin{scriptsize}DEPARTAMENTO DE MATEM\'ATICAS, UNIVERSIDAD AUT\'ONOMA DE MADRID, 28049, MADRID, SPAIN.\end{scriptsize}\\
\begin{scriptsize}\textit{E-mail address}: \texttt{elias.baro@uam.es}\end{scriptsize}\\

\hspace{-0.6cm}\begin{scriptsize}DEPARTAMENTO DE MATEM\'ATICAS, UNIVERSIDAD AUT\'ONOMA DE MADRID, 28049, MADRID, SPAIN.\end{scriptsize}\\
\begin{scriptsize}\textit{E-mail address}: \texttt{margarita.otero@uam.es}\end{scriptsize}\\

\end{document}